%% file: main.tex
\documentclass{amsart}

\include{Preamble}

\graphicspath{{figures/}}

\begin{document}

\title[Higher-Order Adaptive Methods for Nonlinear Elliptic Equations]{Higher-order Adaptive Finite Difference Methods for Fully Nonlinear Elliptic Equations}

\author{Brittany D. Froese}
\thanks{The first author was partially supported by NSF DMS-1619807.  The second author was partially supported by NSERC Discovery grant RGPIN-2016-03922 and by Funda\c{c}\~ao para a Ci\^encia e Tecnologia (FCT) doctoral grant (SFRH/BD/84041/2012).}
\address{Department of Mathematical Sciences, New Jersey Institute of Technology, University Heights, Newark, New Jersey 07102, USA ({\tt bdfroese@njit.edu})}

\author{Tiago Salvador}
\address{Department of Mathematics and Statistics, McGill University, 805 Sherbrooke Street West, Montreal, Quebec, H3A 0G4, Canada ({\tt tiago.saldanhasalvador@mail.mcgill.ca}).}

\date{\today}

\begin{abstract}
We introduce generalised finite difference methods for solving fully nonlinear elliptic partial differential equations.  Methods are based on piecewise Cartesian meshes augmented by additional points along the boundary.  This allows for adaptive meshes and complicated geometries, while still ensuring consistency, monotonicity, and convergence.  We describe an algorithm for efficiently computing the non-traditional finite difference stencils.  We also present a strategy for computing formally higher-order convergent methods.  Computational examples demonstrate the efficiency, accuracy, and flexibility of the methods.
\end{abstract}



\maketitle

\setcounter{tocdepth}{1}


\input{Introduction}
\input{Schemes}

\input{Quadtrees}
\input{accurate}

\input{Examples}

\input{conclusions}

\bibliographystyle{amsalpha}
\bibliography{bibliography}

\end{document}

%% file: Preamble.tex
\usepackage{amsmath}
\usepackage{amssymb}
\usepackage{amsthm}
\usepackage{dsfont}

\usepackage{graphicx}
\usepackage{subfigure}
\graphicspath{{figures/}}

\usepackage{siunitx}

\usepackage{enumerate}

\usepackage{bbold}

\usepackage[colorlinks = true, linkcolor = blue, citecolor = blue, urlcolor = blue]{hyperref}

\usepackage{mathtools}

\usepackage{algorithm}
\usepackage{algorithmic}

\usepackage{tikz}
\usetikzlibrary{fit}

\usepackage{multirow}

\newtheorem{theorem}{Theorem}[section]

\newtheorem{lemma}[theorem]{Lemma}

\newtheorem{definition}[theorem]{Definition}
\theoremstyle{remark}

\theoremstyle{definition}
\newtheorem{example}{Example}[section]

\newcommand{\R}{\mathbb{R}}

\newcommand{\bO}{\mathcal{O}}
\newcommand{\Af}{\mathcal{A}}
\newcommand{\Nf}{\mathcal{N}}
\newcommand{\Sf}{\mathcal{S}}
\newcommand{\bq}{\begin{equation}}
\newcommand{\eq}{\end{equation}}

\newcommand{\Norm}[1]{\left\Vert#1\right\Vert}
\newcommand{\abs}[1]{\vert#1\vert}
\newcommand{\Abs}[1]{\left\vert#1\right\vert}

\usepackage{xspace}

\DeclareMathOperator{\argmin}{argmin}

\DeclarePairedDelimiter\floor{\lfloor}{\rfloor}

\newcommand{\boundary}{\partial \Omega}

\newcommand{\widthtwofigures}{0.48\textwidth}

\usepackage{caption}

\newcommand{\twocolumnsfigures}[2]{
\begin{figure}
\centering
\begin{minipage}{\widthtwofigures}\centering
#1
\end{minipage}
\hfill
\begin{minipage}{\widthtwofigures}\centering
#2
\end{minipage}
\end{figure}
}

\newcommand{\Dt}{\mathcal{D}}

\newcommand{\G}{\mathcal{G}}


%% file: Introduction.tex
\section{Introduction}\label{sec:intro}

In this article we design generalised finite difference methods for a large class of fully nonlinear degenerate elliptic partial differential equations (PDEs).  The approximation schemes are almost-monotone, which allows us to exploit the Barles-Souganidis convergence framework.  A key feature of these methods is the use of piecewise Cartesian grids, augmented by additional discretisation points along the boundary.  Because of the underlying structure of the grids, the methods we design overcome several hurdles that exist for current numerical methods.  1) The non-standard finite difference stencils can be constructed efficiently.  2) The monotone approximation schemes preserve consistency near the boundary.  3) Higher-order schemes are easily designed.  4) Complicated geometries are easily handled.

\subsection{Background}
Fully nonlinear elliptic partial differential equations (PDEs) arise in numerous applications including reflector/refractor design~\cite{GlimmOlikerReflectorDesign}, meteorology~\cite{Cullen}, differential geometry~\cite{Caf_MAGeom}, astrophysics~\cite{FrischUniv}, seismology~\cite{EFWass},  mesh generation~\cite{Budd}, computer graphics~\cite{Osher_book}, and mathematical finance~\cite{FlemingSoner}.  Realistic applications often involve complicated domains and highly non-smooth data.  Thus the development of robust numerical methods that are compatible with adaptive mesh refinement is a priority.

In recent years, the numerical solution of these equations has received a great deal of attention, and several new methods have been developed including finite difference methods~\cite{BFO_MA,FinnGrid,Loeper,Saumier,SulmanWilliamsRussell}, finite element methods~\cite{Awanou,Bohmer,BrennerNeilanMA2D,Smears}, least squares methods~\cite{DGnum2006}, and methods involving fourth-order regularisation terms~\cite{FengNeilan}.  However, these methods are not designed to compute weak solutions. When the ellipticity of the equation is degenerate or no smooth solution exists, methods become very slow, are unstable, or converge to an incorrect solution.

Using a framework developed by Barles and Souganidis~\cite{BSnum}, provably convergent (monotone) methods have recently been constructed for several fully nonlinear equations using wide finite difference stencils~\cite{ObermanWS}.  Recently, the idea of wide stencil schemes has been adapted to produce monotone approximations for a large class of fully nonlinear elliptic operators posed on very general meshes or points clouds~\cite{FroeseMeshfree}.  However, constructing the necessary finite difference stencils is a very expensive process, and the resulting approximations have only $\bO(\sqrt{h})$ accuracy.

\subsection{Contributions of this work}
The goal of this article is to design higher-order, adaptive, convergent finite difference methods for the solution of the degenerate elliptic PDE
\bq\label{eq:PDE}
F(x,u(x),D^2u(x)) \equiv \max\limits_{\theta\in[0,2\pi)}F_\theta(x,u(x), u_{\theta\theta}(x)) = 0,
\eq
where $u_{\theta\theta}$ denotes the second directional derivative of $u$ in the direction $e_\theta = (\cos\theta,\sin\theta) \in \R^2$.  This includes a wide range of PDE operators including monotone functions of the eigenvalues of the Hessian matrix $D^2u$ and general convex functions of the Hessian matrix~\cite[Proposition~5.3]{Evans_NonlinearElliptic}.  We note that the maxima can also be taken over a subset of directions by simply setting $F_\theta=-1$ for directions that are not active in the PDE operator.  The methods described in this article can also be trivially adapted to include minima and other monotone functions of the operators $F_\theta$.

We remark that the PDE operator can also include Lipschitz continuous dependence on the gradient $\nabla u$.  Monotone approximation of first-order operators is fairly well-established and does not require the wide-stencil structure that is necessary to correctly discretise second-order operators.  In this article, we omit terms involving the gradient for the sake of compactness and clarity.

We take as a starting point the meshfree finite difference approximations developed in~\cite{FroeseMeshfree}.  Given a set of discretisation points $\G$, that work introduced a generalised finite difference approximation of~\eqref{eq:PDE} of the form
\bq\label{eq:approxMeshfree} \max\limits_{\theta\in{\Af}\subset[0,2\pi)}F_\theta\left(x_i, u(x_i), \sum\limits_{j\in\Nf(i,\theta)}a_{i,j,\theta}(u(x_i)-u(x_j))\right) = 0, \quad x_i \in \G. \eq
Above, the set of neighbours $\Nf(i,\theta)$ gives indices of several points that are near $x_i$ and align closely with the $e_\theta$ direction.  In the original paper, these sets of neighbours are computed through brute force, by finding and inspecting all nodes lying within a distance $\sqrt{h}$ of each other.

In this article, we describe the construction of piecewise Cartesian meshes using a quadtree structure, which is augmented by additional discretisation points along the boundary in order to preserve consistency.  We demonstrate that the resulting set of discretisation points satisfies the conditions required by~\cite[Theorem~13]{FroeseMeshfree}.  This ensures that it is possible to build monotone (convergent) approximations.  Moreover, the structure of these quadtrees allows for easy mesh adaptation, with refinement criteria that can either be specified \emph{a priori} or determined automatically from the quality of the solution of~\eqref{eq:approxMeshfree}.

By exploiting the underyling structure of the quadtree meshes, we design an efficient method for constructing the set of neighbours $\Nf(i,\theta)$ required by our generalised finite difference stencils.  This leads to a much faster numerical method for approximating solutions of~\eqref{eq:PDE} that can easily handle singular solutions, complicated geometries, and highly non-uniform meshes.

Finally, we describe a strategy for producing higher-order almost-monotone methods using the framework of~\cite{FOFiltered}.  This requires combining the monotone scheme with a formally higher-order approximation.  By utilising the structure and symmetry within the quadtree mesh, we obtain a simple strategy for constructing higher-order schemes in the interior of the domain.  Near the boundary, where these simple schemes no longer exist, we propose a least-squares approach to constructing higher-order schemes.

\subsection{Contents}
In \autoref{sec:schemes}, we review the theory of generalised finite difference approximations for fully nonlinear elliptic equations.  In \autoref{sec:meshes}, we describe our strategy for constructing meshes and finite difference stencils. In \autoref{sec:accurate}, we describe a higher-order implementation of these methods. In \autoref{sec:examples}, we present several computational examples.  Finally, in \autoref{sec:conclusions}, we provide conclusions and perspectives.

%% file: Schemes.tex
\section{Generalised Finite Difference Schemes}\label{sec:schemes}

In this section we review existing results on the construction and convergence of numerical methods for fully nonlinear elliptic equations.

\subsection{Weak solutions}
One of the challenges associated with the approximation of fully nonlinear PDEs is the fact that classical (smooth) solutions may not exist.  It thus becomes necessary to interpret PDEs using some notion of weak solution, and the numerical methods that are used need to respect this notion of weak solution.  The most common concept of weak solution for this class of PDEs is the \emph{viscosity solution}, which involves transferring derivatives onto smooth test functions via a maximum principle argument~\cite{CIL}.

The PDEs we consider in this work belong to the class of degenerate elliptic equations,
\bq\label{eq:PDEelliptic} F(x,u(x),D^2u(x)) = 0, \quad x\in\bar{\Omega}\subset\R^2.\eq
\begin{definition}[Degenerate elliptic]\label{def:elliptic}
The operator
$F:\bar{\Omega}\times\R\times\Sf^2\to\R$
is \emph{degenerate elliptic} if 
\[ F(x,u,X) \leq F(x,v,Y) \]
whenever $u \leq v$ and $X \geq Y$.
\end{definition}

We note that the operator is also defined on the boundary $\partial\Omega$ of the domain, which allows both the PDE and the boundary conditions to be contained within equation~\eqref{eq:PDEelliptic}.  For example, if Dirichlet data $u(x) = g(x)$ is given on the boundary, the elliptic operator at the boundary will be defined by
\[ F(x,u(x),D^2u(x)) = u(x) - g(x), \quad x \in \partial\Omega. \]

The PDE operators~\eqref{eq:PDE} that we consider in this work are degenerate elliptic if they are non-decreasing functions of their second argument ($u$) and non-increasing functions of all subsequent arguments (which involve second directional derivatives).

Since degenerate elliptic equations need not have classical solutions, solutions need to be interpreted in a weak sense.  The numerical methods developed in this article are guided by the very powerful concept of the viscosity solution~\cite{CIL}.  Checking the definition of the viscosity solution requires checking the value of the PDE operator for smooth test functions lying above or below the semi-continuous envelopes of the candidate solution.

\begin{definition}[Upper and Lower Semi-Continuous Envelopes]\label{def:envelope}
The \emph{upper and lower semi-continuous envelopes} of a function $u(x)$ are defined, respectively, by
\[ u^*(x) = \limsup_{y\to x}u(y), \]
\[ u_*(x) = \liminf_{y\to x}u(y). \]
\end{definition}

\begin{definition}[Viscosity subsolution (supersolution)]\label{def:subsuper}
An upper (lower) semi-continuous function $u$ is a \emph{viscosity subsolution (supersolution)} of~\eqref{eq:PDE} if for every $\phi\in C^2(\bar{\Omega})$, whenever $u-\phi$ has a local maximum (minimum)  at $x \in \bar{\Omega}$, then
\[ 
F_*^{(*)}(x,u(x),D^2\phi(x)) \leq (\geq)  0 .
\]
\end{definition}
\begin{definition}[Viscosity solution]\label{def:viscosity}
A function $u$ is a \emph{viscosity solution} of~\eqref{eq:PDE} if $u^*$ is a subsolution and $u_*$ a supersolution.
\end{definition}


An important property of many elliptic equations is the comparison principle, which immediately implies uniqueness of the solution.
\begin{definition}[Comparison principle]\label{def:comparison}
A PDE has a \emph{comparison principle} if whenever $u$ is an upper semi-continuous subsolution and $v$ a lower semi-continuous supersolution of the equation, then $u \leq v$ on $\bar{\Omega}$.
\end{definition}

\subsection{Convergence of elliptic schemes}

In order to construct convergent approximations of elliptic operators, we will rely on the framework provided by Barles and Souganidis~\cite{BSnum} and further developed by Oberman~\cite{ObermanDiffSchemes}.

We consider finite difference schemes that have the form
\bq\label{eq:approx} F^\epsilon(x,u(x),u(x)-u(\cdot)) = 0 \eq
where $\epsilon$ is a small parameter.

The convergence framework requires notions of consistency and monotonicity, which we define below.

\begin{definition}[Consistency]\label{def:consistency}
The scheme~\eqref{eq:approx} is \emph{consistent} with the equation~\eqref{eq:PDE} if for any smooth function $\phi$ and $x\in\bar{\Omega}$,
\[ \limsup_{\epsilon\to0^+,y\to x,\xi\to0} F^\epsilon(y,\phi(y)+\xi,\phi(y)-\phi(\cdot)) \leq F^*(x,\phi(x),\nabla\phi(x),D^2\phi(x)), 
\]
\[ \liminf_{\epsilon\to0^+,y\to x,\xi\to0} F^\epsilon(y,\phi(y)+\xi,\phi(y)-\phi(\cdot)) \geq F_*(x,\phi(x),\nabla\phi(x),D^2\phi(x)). \]
\end{definition}

\begin{definition}[Monotonicity]\label{def:monotonicity}
The scheme~\eqref{eq:approx} is monotone if $F^\epsilon$ is a non-decreasing function of its final two arguments.
\end{definition}

Schemes that satisfy these two properties respect the notion of the viscosity solution at the discrete level.  In particular, these schemes preserve the maximum principle and are guaranteed to converge to the solution of the underlying PDE.

\begin{theorem}[Convergence~\cite{ObermanDiffSchemes}]\label{thm:convergeVisc}
Let $u$ be the unique viscosity solution of the PDE~\eqref{eq:PDE}, where $F$ is a degenerate elliptic operator with a comparison principle.  Let the finite difference approximation $F^\epsilon$ be consistent and monotone and let $u^\epsilon$ be any solution of the scheme~\eqref{eq:approx}, with bounds independent of $\epsilon$.  Then $u^\epsilon$ converges uniformly to $u$ as $\epsilon\to0$.
\end{theorem}

The above theorem assumes existence of a bounded solution to the approximation scheme.  This is typically straightforward to show for a consistent, monotone approximation of a well-posed PDE, though the precise details can vary slightly and rely on available well-posedness theory for the PDE in question.  

\begin{theorem}[Existence and Stability~{\cite[Lemmas~35-36]{FroeseGauss}}]\label{thm:stability}
Let $F^\epsilon$ be a consistent, monotone scheme that is Lipschitz in its last two arguments.  Suppose also that there exist strict classical sub- and super-solutions to the PDE~\eqref{eq:PDEelliptic}.  Then for small enough $\epsilon>0$, the scheme~\eqref{eq:approx} has a solution~$u^\epsilon$.  Moreover, there exists a constant $M>0$ such that $\|u^\epsilon\|_\infty \leq M$ for sufficiently small $\epsilon>0$.
\end{theorem}

In many cases, simple quadratic functions will serve as the sub- and super-solutions required by Theorem~\ref{thm:stability}.  For more complicated PDE operators, particularly those with a non-trivial dependence on the gradient $\nabla u$, the theory of classical solutions of the equation can often be used to show the existence of these sub- and super-solutions.

\subsection{Meshfree finite difference approximations}

In~\cite{FroeseMeshfree}, a new generalised finite difference method was introduced for approximating fully nonlinear second order elliptic operators on point clouds.  We review the key results of that work, which will be foundational to the higher-order adaptive methods that will be developed in the remainder of this article.

\begin{definition}[Notation]\label{def:notation}
\end{definition}
\begin{enumerate}
\item[(N1)] $\Omega\subset\R^2$ is a bounded domain with Lipschitz boundary $\partial\Omega$. 
\item[(N2)] $\G\subset\bar{\Omega}$ is a point cloud consisting of the points $x_i$, $i=1,\ldots,N$.
\item[(N3)] $h = \sup\limits_{x\in{\Omega}}\min\limits_{y\in\G}\abs{x-y}$ is the spatial resolution of the point cloud.  In particular, every ball of radius $h$ contained in $\bar{\Omega}$ contains at least one discretisation point $x_i$.
\item[(N4)] $h_B = \sup\limits_{x\in{\partial\Omega}}\min\limits_{y\in\G\cap\partial\Omega}\abs{x-y}$ is the resolution of the point cloud on the boundary.  In particular, every ball of radius $h_B$ centred at a boundary point $x\in\partial\Omega$ contains at least one discretisation point $x_i \in \G\cap\partial\Omega$ on the boundary.
\item[(N5)] $\delta = \min\limits_{x\in\Omega\cap\G}\inf\limits_{y\in\partial\Omega}\abs{x-y}$ is the distance between the set of interior discretisation points and the boundary.  In particular, if $x_i\in\G\cap\Omega$ and $x_j\in\partial\Omega$, then the distance between $x_i$ and $x_j$ is at least $\delta$.
\item[(N6)] $d\phi$ is the angular resolution used to approximate the second directional derivatives $u_{\theta\theta}$.
\item[(N7)] $d\theta$ is the angular resolution used to approximate the nonlinear operator.
\item[(N8)] $\epsilon$ is the search radius associated with the point cloud.
\end{enumerate}

Discretising the PDE requires approximating second directional derivatives $u_{\theta\theta}$ at each interior discretisation point $x_i\in\G$.  To accomplish this, we consider all points $x_j\in\G\cap B(x_i,\epsilon)$ within a search neighbourhood of radius $\epsilon$ centred at $x_i$.  Discretisation points within this neighbourhood can be written in polar coordinates $(r,\phi)$ with respect to the axes defined by the lines $x_0 + t(\cos\theta,\sin\theta)$, $x_0 + t(-\sin\theta,\cos\theta)$.  We seek one neighbouring discretisation point in each quadrant described by these axes, with each neighbour aligning as closely as possible with the line $x_0 + t\nu$, where $\nu = (\cos\theta,\sin\theta)$.  That is, we select the neighbours
\bq\label{eq:neighbours} x_j \in \argmin\left\{{\sin^2\phi} \mid (r,\phi)\in\G^h\cap B(x_0,\epsilon) \text{ is in the $j$th quadrant}\right\}\eq
for $j = 1, \ldots, 4$.  See Figure~\ref{fig:stencil}. {We say that a stencil with angular resolution $d\phi$ exists for the point cloud $\G$ if for all interior discretisation points, the four discretisation points $x_j \in \G$ defined by~\eqref{eq:neighbours} exist and satisfy $d\phi = \max \{\phi_j\}$.} 

Because of the ``wide-stencil'' nature of these approximations (since the search radius $\epsilon \gg h$), care must be taken near the boundary.  In order to preserve consistency up to the boundary, it is necessary that the boundary be more highly resolved than the interior ($h_B \ll h$).  In particular, this means that a simple Cartesian mesh (or piecewise Cartesian mesh) is \emph{not} sufficient for producing consistent schemes up to the boundary.

\begin{figure}[htp]
\centering
\subfigure[]{
\includegraphics[width=0.4\textwidth]{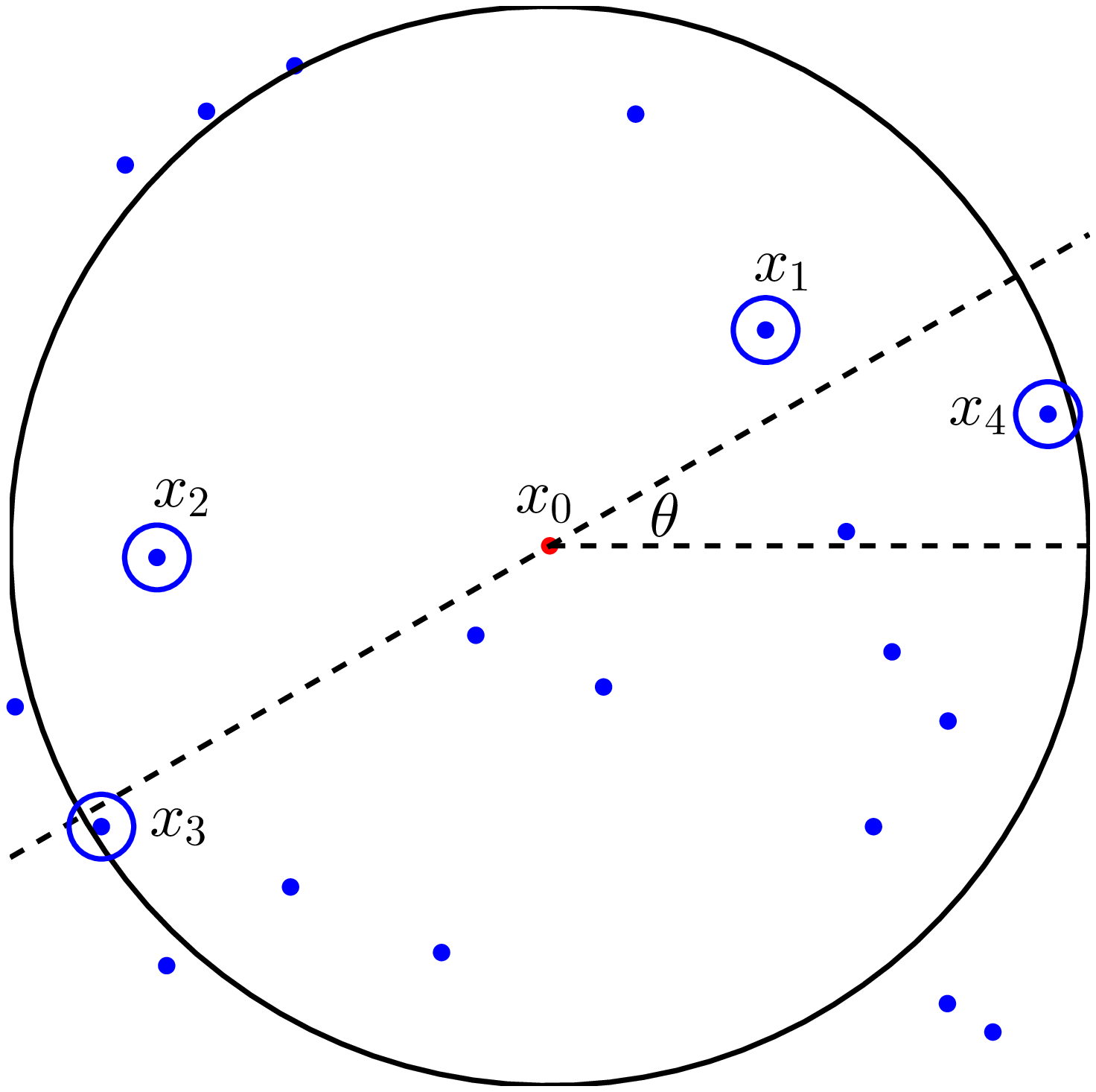}\label{fig:stencil1}}
\subfigure[]{
\includegraphics[width=0.47\textwidth]{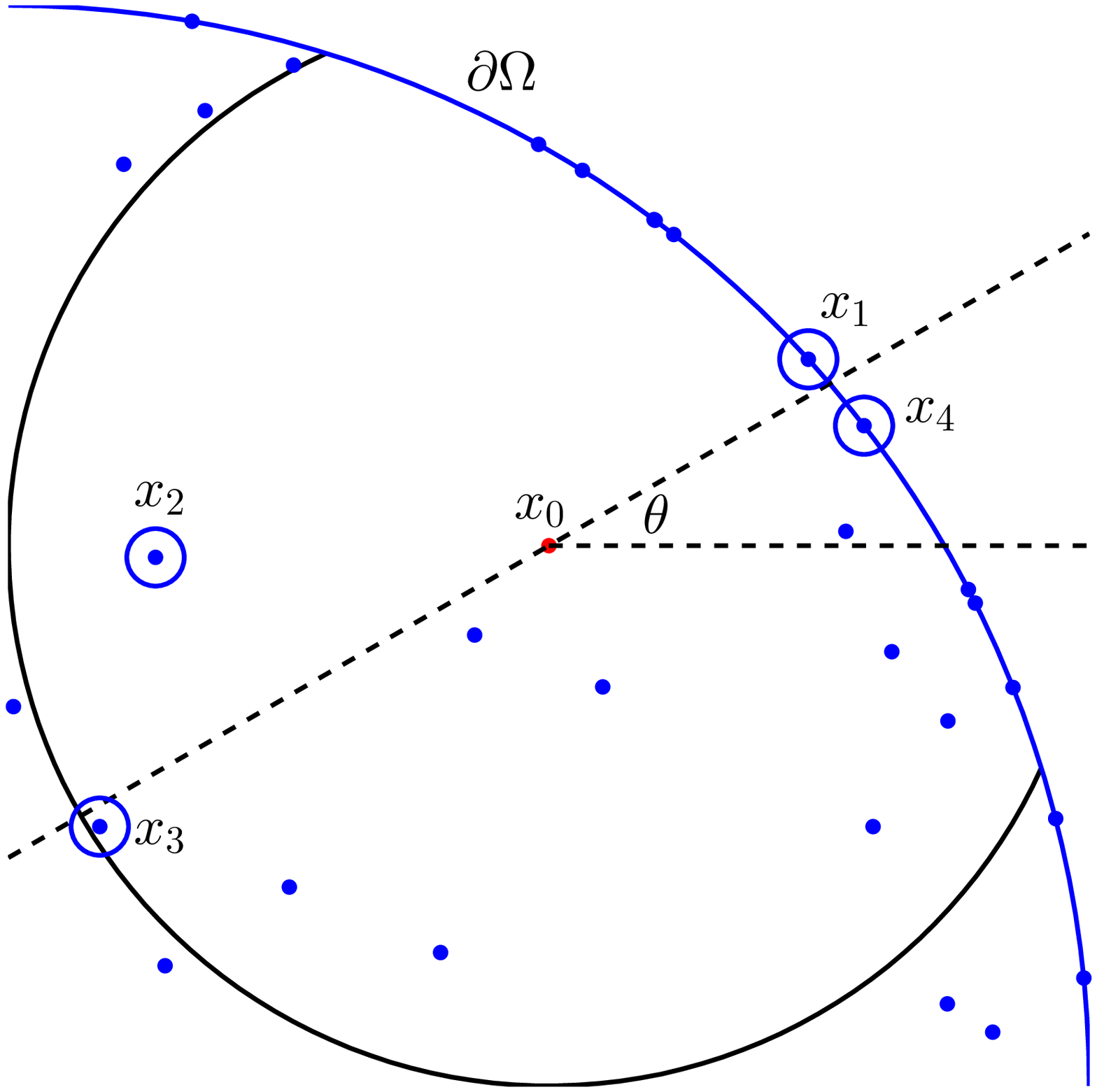}\label{fig:stencil2}}
\caption{A finite difference stencil chosen from a point cloud \subref{fig:stencil1}~in the interior and \subref{fig:stencil2}~near the boundary.}
\label{fig:stencil}
\end{figure}

Then a consistent, monotone approximation of $u_{\theta\theta}$ is
\[ \Dt_{\theta\theta}^hu(x_0) = \sum\limits_{j=1}^4 a_j(u(x_j)-u(x_0)) \]
where we use the polar coordinate characterisation of the neighbours to define
\[ S_j =  r_j\sin\phi_j, \quad C_j = r_j\cos\phi_j\]
and the coefficients are given by
\[\begin{split}
a_1 &= \frac{2S_4(C_3S_2-C_2S_3)}{(C_3S_2-C_2S_3)(C_1^2S_4-C_4^2S_1)-(C_1S_4-C_4S_1)(C_3^2S_2-C_2^2S_3)}\\
a_2 &= \frac{2S_3(C_1S_4-C_4S_1)}{(C_3S_2-C_2S_3)(C_1^2S_4-C_4^2S_1)-(C_1S_4-C_4S_1)(C_3^2S_2-C_2^2S_3)}\\
a_3 &= \frac{-2S_2(C_1S_4-C_4S_1)}{(C_3S_2-C_2S_3)(C_1^2S_4-C_4^2S_1)-(C_1S_4-C_4S_1)(C_3^2S_2-C_2^2S_3)}\\
a_4 &= \frac{-2S_1(C_3S_2-C_2S_3)}{(C_3S_2-C_2S_3)(C_1^2S_4-C_4^2S_1)-(C_1S_4-C_4S_1)(C_3^2S_2-C_2^2S_3)}.
\end{split}
\]

In general, the PDE requires evaluating second directional derivatives in all possible directions.  Instead, we consider a finite subset $\Af=\left\{jd\theta\mid j = 0, \ldots, \floor{\frac{2\pi}{d\theta}}\right\}\subset[0,2\pi)$ with a resolution $d\theta$.

Then we can substitute these coefficients into~\eqref{eq:approxMeshfree} to obtain the scheme:
\bq\label{eq:approxMeshfree2} F_i[u] \equiv \max\limits_{\theta\in{\Af}}F_\theta\left(x_i, u(x_i), \sum\limits_{j\in\Nf(i,\theta)}a_{i,j,\theta}(u(x_i)-u(x_j))\right) = 0, \quad x_i \in \G. \eq

We recall the convergence result from~\cite[Theorem~18]{FroeseMeshfree}.
\begin{theorem}[Convergence]\label{thm:convergenceStencils}
Let $F$ be a degenerate elliptic operator with a comparison principle that is Lipschitz continuous in $u_{\theta\theta}$ for each $\theta\in[0,2\pi)$ and let $u$ be the unique viscosity solution of the PDE~\eqref{eq:PDE}. Suppose also that~\eqref{eq:PDE} has a strict classical sub- and super-solution.  Consider a sequence of point clouds $\G^n$, with parameters defined as in Definition~\ref{def:notation}, which satisfy the following conditions.
\begin{itemize}
\item The spatial resolution $h^n\to0$ as $n\to\infty$.
\item The boundary resolution satisfies $h_B^n/\delta^n \to 0$ as $n\to\infty$.
\item The search radius satisfies both $\epsilon^n\to0$ and $h^n/\epsilon^n\to0$ as $n\to\infty$.
\item The angular resolution $d\theta^n\to0$ as $h^n\to0$.
\end{itemize}
Then for sufficiently large $n$, the approximation scheme~\eqref{eq:approxMeshfree2} admits a solution~$u^n$ and $u^n$ converges uniformly to $u$ as $n\to\infty$.
\end{theorem}

We note that the angular resolution that emerges from the scheme satisfies $d\phi = \bO(\max\{h/\epsilon,h_B/\delta\})$ (Figure~\ref{fig:MeshfreeAngular}).  For a uniform grid, a natural choice of parameters is $\epsilon = \bO(\sqrt{h})$, $h_B = \bO(h^{3/2})$, $\delta = \bO(h)$, $d\theta = \bO(\sqrt{h})$.  This leads to a formally optimal discretisation error of $\bO(\sqrt{h})$.

We remark also that these parameters can be defined locally instead of globally in order to accommodate highly non-uniform meshes.

\begin{figure}
\centering
\subfigure[]{\includegraphics[width=\widthtwofigures]{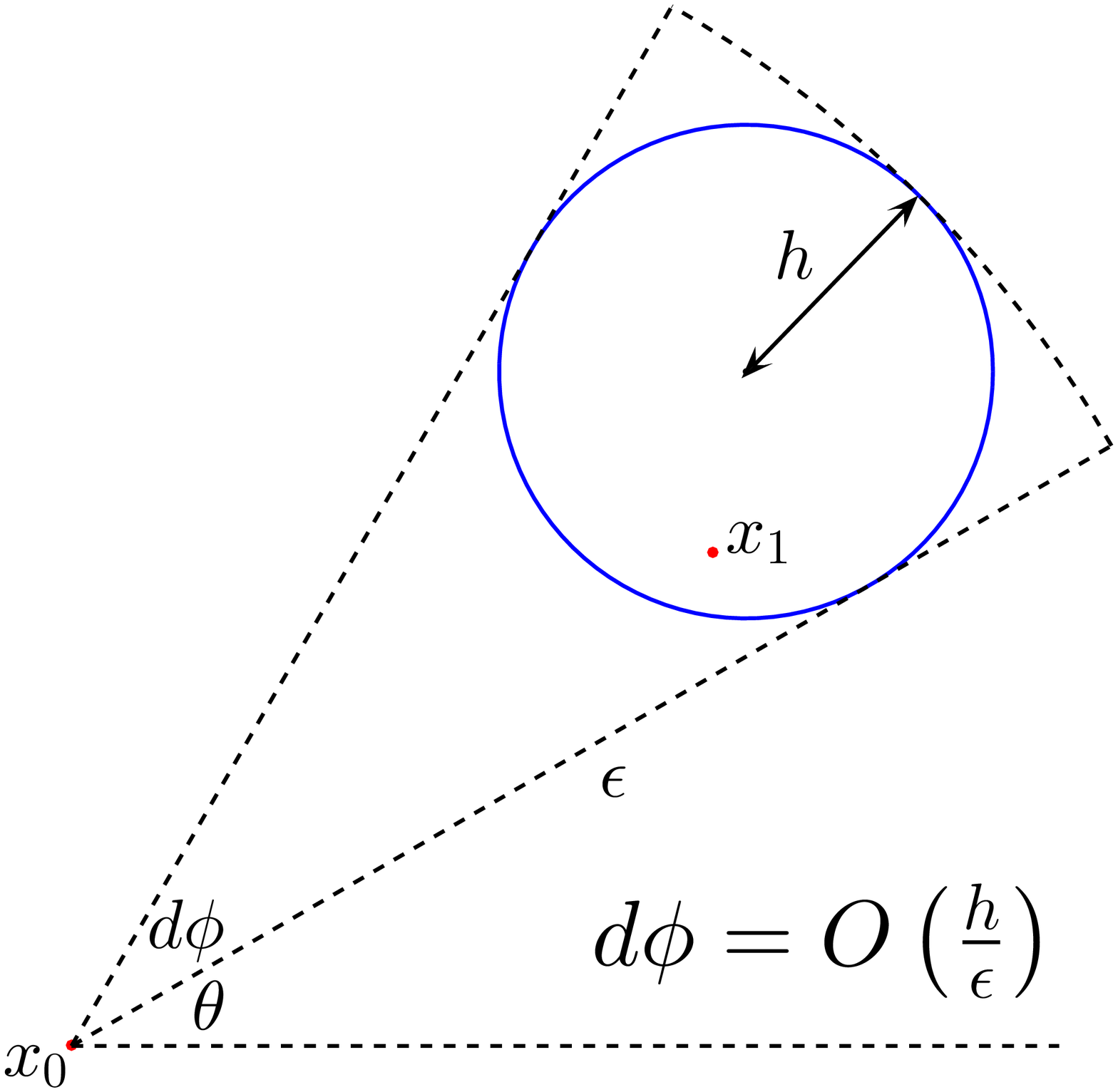}}\label{fig:MeshfreeExistence}
\subfigure[]{\includegraphics[width=\widthtwofigures]{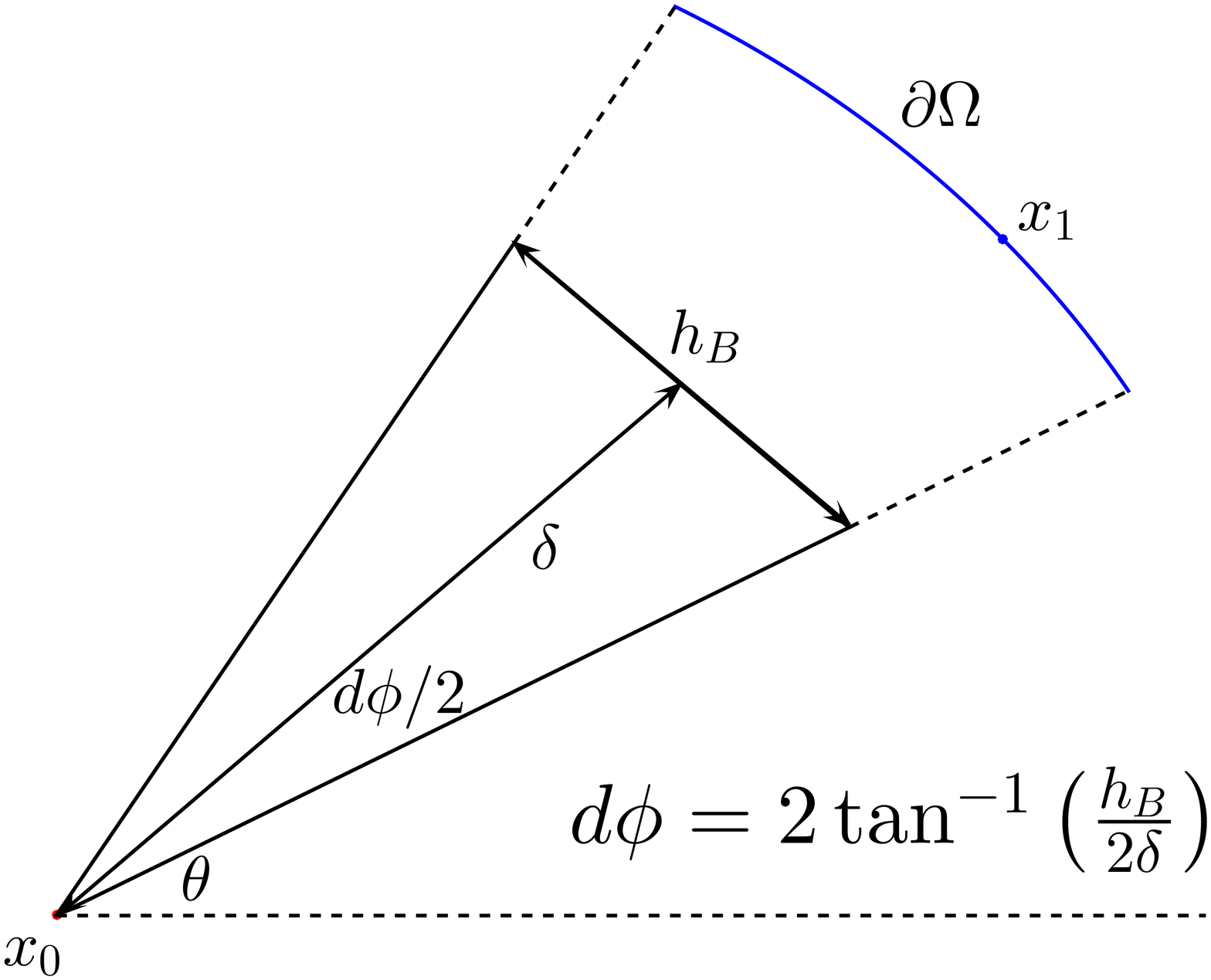}}\label{fig:MeshfreeBoundaryExist}
\caption{The angular resolution of a generalised finite difference stencil.}
\label{fig:MeshfreeAngular}
\end{figure}

%% file: Quadtrees.tex
\section{Construction of Meshes and Stencils}\label{sec:meshes}

In this section, we explain how we use augmented quadtrees~\cite[Chapter 14]{bookquadtree} to build piecewise Cartesian meshes with additional discretisation points on the boundary. We organise this section as follows. In \autoref{sec:quadtree}, we recall the basic structure of a quadtree.  In \autoref{sec:meshes_boundary} we explain how we augment the quadtree to deal with the boundary. In \autoref{sec:meshes_stencilsearch}, we explain how the quadtrees are used to efficiently find the stencils. Finally, in \autoref{sec:meshes_adaptivity} we discuss mesh adaptation.

\subsection{Quadtrees}\label{sec:quadtree}

Quadtrees are based on a simple idea: a square can be divided into four smaller squares which correspond to the four quadrants of the square. A quadtree is then a rooted tree in which every internal node has four children and every node in the tree corresponds to a square. A square with no children is called a leaf square. See Figure \ref{fig:quadtree}.

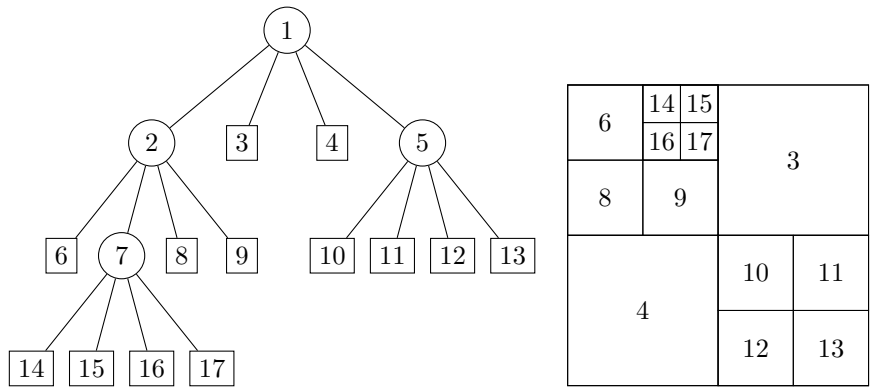
\begin{figure}[h]
\centering

\begin{tabular}{cc}
\begin{tikzpicture}[every node/.style={rectangle,draw},level 1/.style={sibling distance=12mm},level 2/.style={sibling distance=8mm},edge from parent/.style={draw, edge from parent path={(\tikzparentnode) -- (\tikzchildnode)}}
]
\node[draw,circle]{1}
child { node[draw,circle] {2} 
child { node {6} }
child { node[draw,circle] {7} child { node {14} } child { node {15} } child { node {16}} child { node {17}}}
child { node {8} }
child { node {9} }}
child { node {3} }
child { node {4} }
child { node[draw,circle] {5} child {node {10}} child {node {11}} child {node {12}} child {node {13}} }
;
\end{tikzpicture} & 
\begin{tikzpicture}[scale=0.5]
\draw (0,0) rectangle (8,8); 
\draw (0,4) rectangle (4,8); 
\draw (4,4) rectangle node{3} (8,8); 
\draw (0,0) rectangle node{4} (4,4); 
\draw (4,0) rectangle (8,4); 
\draw (0,6) rectangle node{6} (2,8); 
\draw (2,6) rectangle (4,8); 
\draw (0,4) rectangle node{8} (2,6); 
\draw (2,4) rectangle node{9} (4,6); 
\draw (4,2) rectangle node{10} (6,4); 
\draw (6,2) rectangle node{11} (8,4); 
\draw (4,0) rectangle node{12} (6,2); 
\draw (6,0) rectangle node{13} (8,2); 
\draw (2,7) rectangle node{14} (3,8); 
\draw (3,7) rectangle node{15} (4,8); 
\draw (2,6) rectangle node{16} (3,7); 
\draw (3,6) rectangle node{17} (4,7); 
\end{tikzpicture}
\end{tabular}
\caption{A quadtree and the corresponding subdivision. The internal nodes are represented with circles and the leaves with squares.}
\label{fig:quadtree}
\end{figure}

Quadtrees can then easily be used to build uniform and non-uniform meshes: the squares' vertices are the mesh points. This structure is appealing because it is general enough to allow for local mesh adaptation, while still maintaining enough structure to efficiently build the finite difference stencils. Indeed, as we will see in \autoref{sec:meshes_stencilsearch}, the quadtree structure allow us to significantly reduce the number of mesh points inspected when constructing our stencils. However, the quadtree in and of itself is not ideal for handling complicated geometries as the mesh points are restricted to be vertices of the squares. 
We observe that the global spatial resolution $h$ of a quadtree corresponds to the length-scale of the largest leaf square.  However, the quadtree can be highly non-uniform and the local spatial resolution near a particular point may be much less than $h$.

\subsection{Meshing the boundary}\label{sec:meshes_boundary}
Quadtrees alone are not enough for the schemes proposed here: the boundary requires additional treatment. As discussed in \autoref{sec:schemes}, the boundary must be more highly resolved than the interior to maintain consistency of the numerical method. As a result, we cannot restrict the mesh points to be the vertices of the squares in the quadtree. 

To overcome this, we build augmented quadtrees: each leaf square that intersects the boundary is marked as such and additional mesh points that lie on the boundary are added and associated with this boundary leaf square. The immediate advantage of this approach is that mesh points may lie exactly on the boundary, which allows us to handle complicated geometries with ease. In addition, by keeping track of which boundary leaf square the mesh points belong to, we preserve one of the key properties of the quadtree: knowledge of the relative position of the mesh points. This allows for efficient construction of the finite difference stencils.

We make the following general assumption: each edge of a leaf square intersects the boundary at most once. This is a reasonable assumption that simply entails that our quadtree must be sufficiently refined near the boundary. See Figure~\ref{fig:MeshesQuadtrees}, where each edge of the grey squares intersects the boundary at most once.

The only question left to address is exactly how many additional boundary mesh points one must add to guarantee the existence of a consistent stencil. This is addressed in Theorem~\ref{thm:convergenceStencils}, which requires that the boundary resolution go to zero more quickly than the resolution of the ``standard'' quadtree, $h_B = o(h)$, and more quickly than the gap between the boundary and the interior nodes, $h_B = o(\delta)$.  Note that these conditions need only be satisfied locally rather than globally.

We define a simple algorithm that enlarges a given point cloud so that the condition $h_B \leq 2\delta \tan(d\theta/2)$ is satisfied locally (see Algorithm \ref{alg:mesh_boundary}).  This ensures that the angular resolution of the finite difference approximations is commensurate with the angular resolution used to approximate the nonlinear operator: $d\phi \lesssim 2\tan^{-1}\left(\dfrac{h_B}{2\delta}\right) \leq d\theta$~(Figure~\ref{fig:MeshfreeAngular}).

\begin{algorithm}[h]
\caption{Building augmented quadtrees}
\label{alg:mesh_boundary}
\begin{algorithmic}[1]
\FOR{each boundary leaf square $S$}
\STATE Add the points in $\partial S \cap \boundary$ to the point cloud $\G$.
\STATE Compute $X = \Omega \cap \G \cap S$.
\STATE Compute $\delta = \min_{x \in X} \min_{y \in \boundary \cap S} \Abs{x-y}$.
\STATE Compute the arc length, $l$, of the curve $\partial \Omega \cap S$.
\STATE Compute the desired boundary resolution $h_B = 4\delta\tan(d\theta/2)$.
\STATE Select $\lceil l/h_B\rceil$ points lying on the curve $\partial \Omega \cap S$.
\STATE Add these points to the point cloud $\G$.
\ENDFOR
\end{algorithmic}
\end{algorithm}


In Figure~\ref{fig:MeshesQuadtrees}, the meshes obtained by applying Algorithm~\ref{alg:mesh_boundary} to the point cloud obtained from a quadtree of depth $4$ are displayed. The fan-shaped domain illustrates the advantages of the local criteria: the boundary is only highly resolved when there are interior mesh points nearby.

%

These augmented quadtrees enable us to construct convergent (consistent and monotone) finite difference approximations.

\begin{lemma}[Approximation with augmented quadtrees]\label{lem:augmentedQuadtree}
Consider a sequence of augmented quadtrees $\G^n$ constructed via Algorithm~\ref{alg:mesh_boundary} with spatial resolution $h_n\to0$. Consider also a sequence of search radii $\epsilon^n = \bO(\sqrt{h^n})$ and angular resolutions $d\theta^n=\bO(\sqrt{h^n})$. 
Then $\G^n$ satisfies the hypotheses of Theorem~\ref{thm:convergenceStencils}.
\end{lemma}

\begin{proof}
We need only verify that $h_B^n/\delta^n \to 0$; the remaining conditions of Theorem~\ref{thm:convergenceStencils} are trivially satisfied.  

We recall that both $h_B^n$ and $\delta^n$ can be defined locally.  Indeed, for each boundary leaf square $S$ we can let
	\[h_{B,S}^n = \sup\limits_{x\in{\partial\Omega}\cap S}\min\limits_{y\in\G^n\cap\partial\Omega \cap S}\abs{x-y}, \quad
	\delta_{S}^n = \min\limits_{x\in\Omega\cap\G^n}\inf\limits_{y\in\partial\Omega \cap S}\abs{x-y}.\]
By construction, Algorithm~\ref{alg:mesh_boundary} ensures that $h_{B,S}^n/\delta_S^n = \bO(d\theta_S^n) \to 0$.

Moreover, it is sufficient to verify these conditions at boundary leaf squares; other interior squares will produce larger values of $\delta^n_S$ and smaller values of $h_{B,S}^n/\delta^n_S$.
\end{proof}

\begin{figure}[h]
\centering
\subfigure[]{\includegraphics[width=\widthtwofigures]{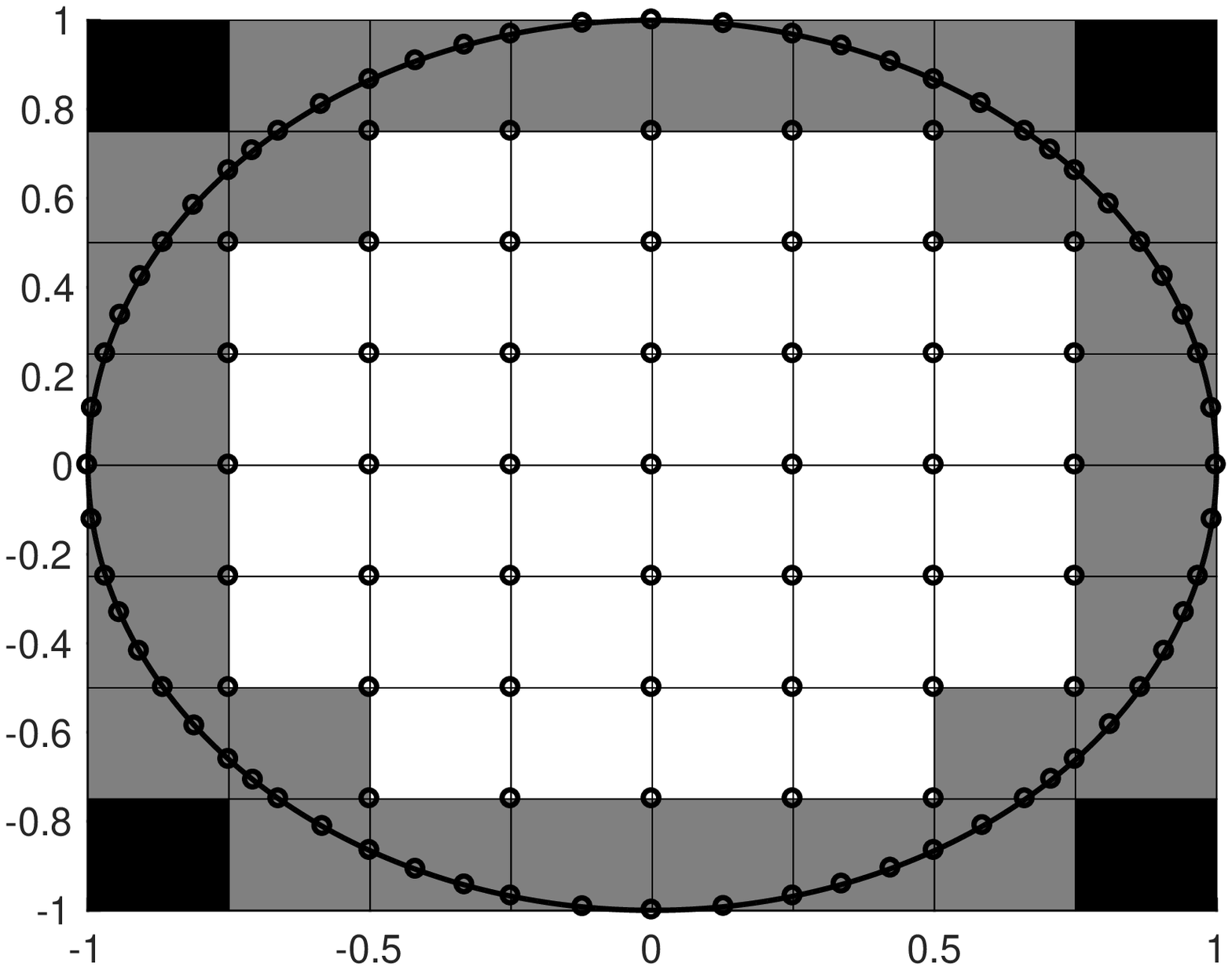}}\label{fig:UniformQuadtreeCircle}
\subfigure[]{\includegraphics[width=\widthtwofigures]{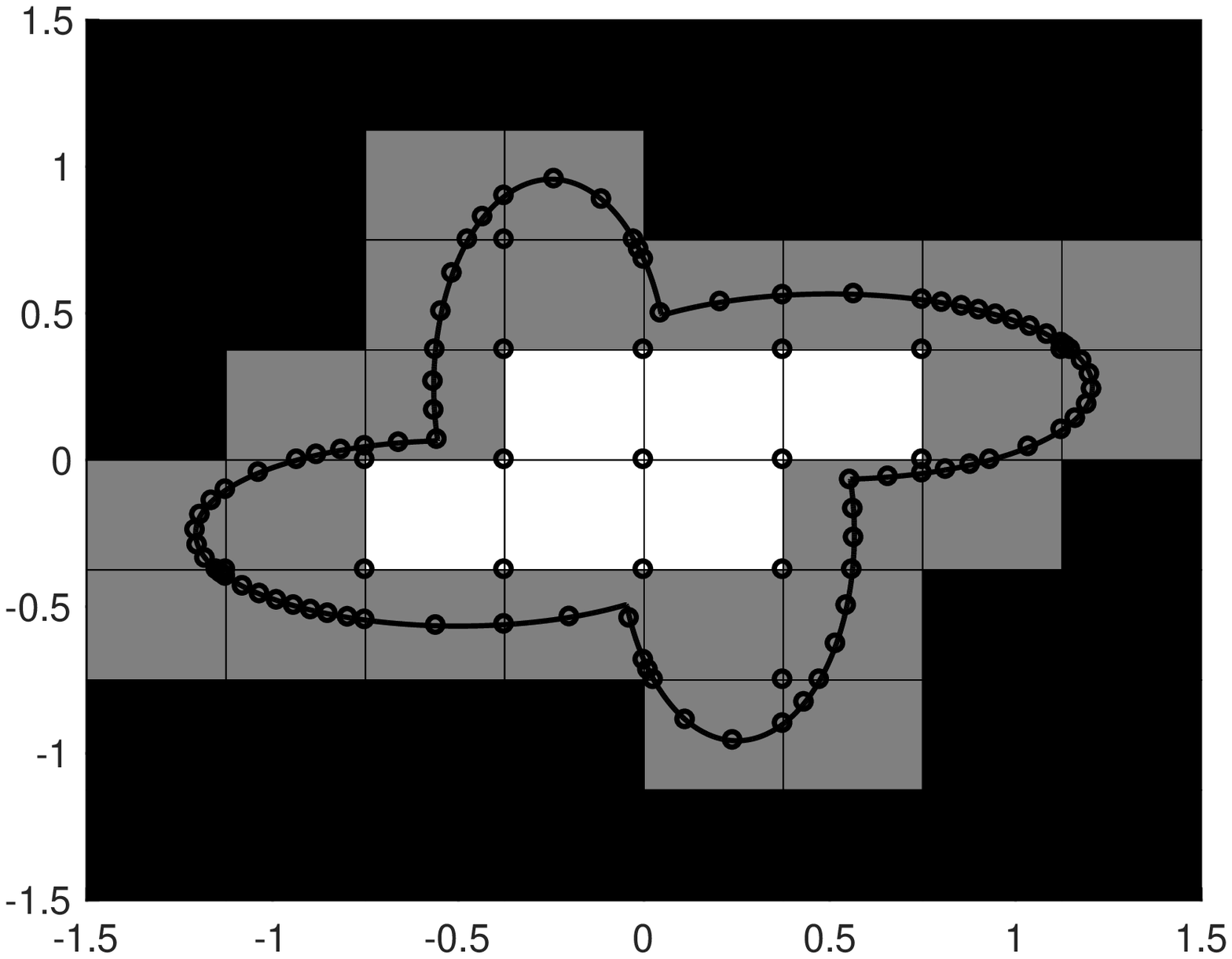}}\label{fig:UniformQuadtreeFan}
\caption{Black squares are part of the quadtree but not used since they are not inside the domain. Grey squares intersect the boundary. White squares are inside the domain.}
\label{fig:MeshesQuadtrees}
\end{figure}


\subsection{Refinement, adaptivity and balance}\label{sec:meshes_adaptivity}

The use of quadtrees also provides a natural means of doing mesh adaptation. A refinement criteria can either be specified \emph{a priori} or determined automatically from the quality of the solution.

In Figure \ref{fig:RefinementQuadtree}, we provide an example of \emph{a priori} refinement: the mesh is refined near the corners of the domain.

Simply refining the quadtree can lead to a very unbalanced quadtree when large squares adjoin several smaller squares. This is an undesirable property for our meshes as it makes the construction of high-order schemes significantly more difficult. Therefore, we always maintain a balanced quadtree: any two neighbouring squares differ by at most a factor of two in length scale (see Figure \ref{fig:RefinementQuadtree}). Balancing a quadtree can be done efficiently; we refer to~\cite[Theorem 14.4]{bookquadtree} for details. 

\begin{figure}[h]
\centering
\subfigure[]{\includegraphics[width=0.45\textwidth]{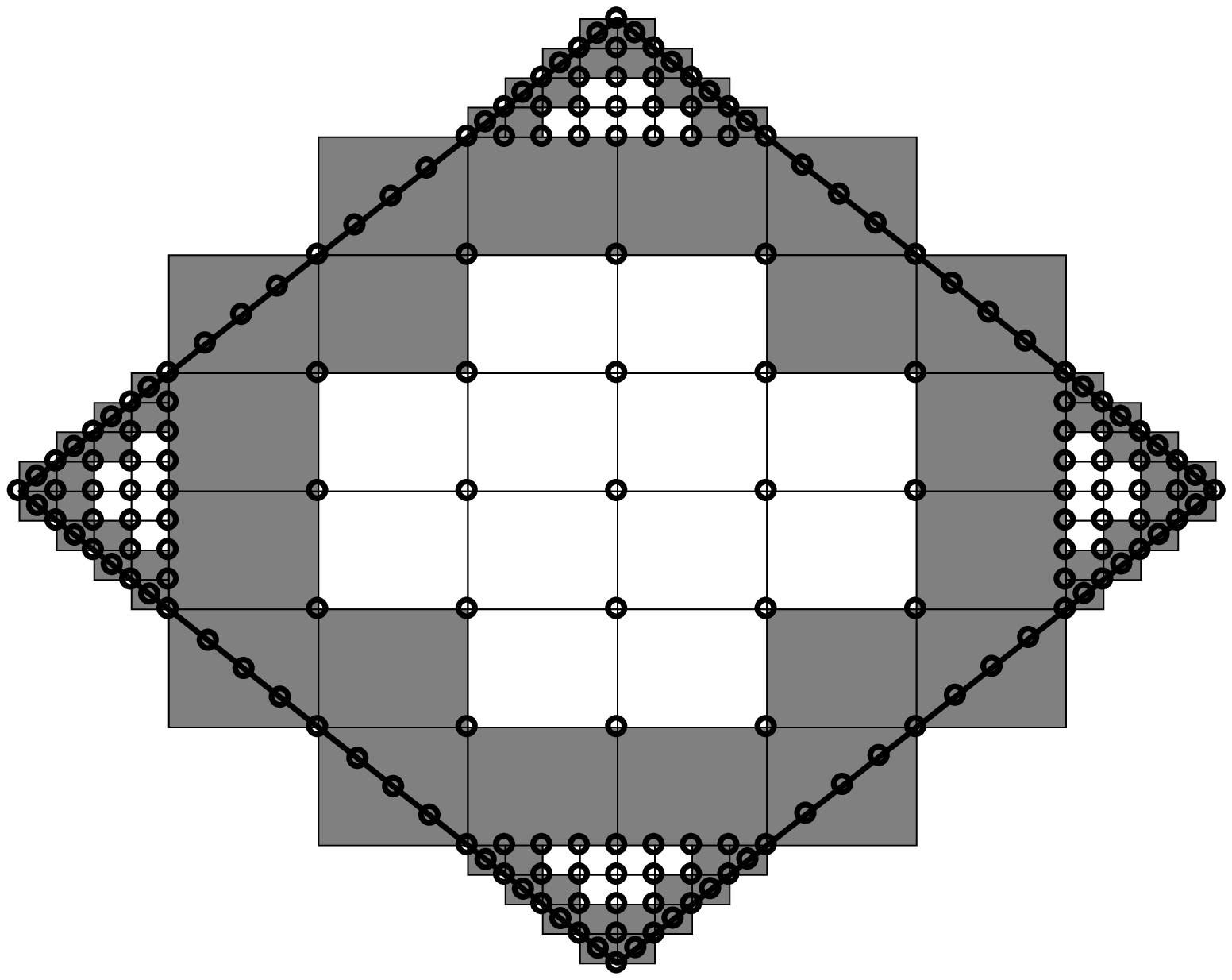}\label{fig:unbalanced}}
\hspace{0.1in}
\subfigure[]{\includegraphics[width=0.45\textwidth]{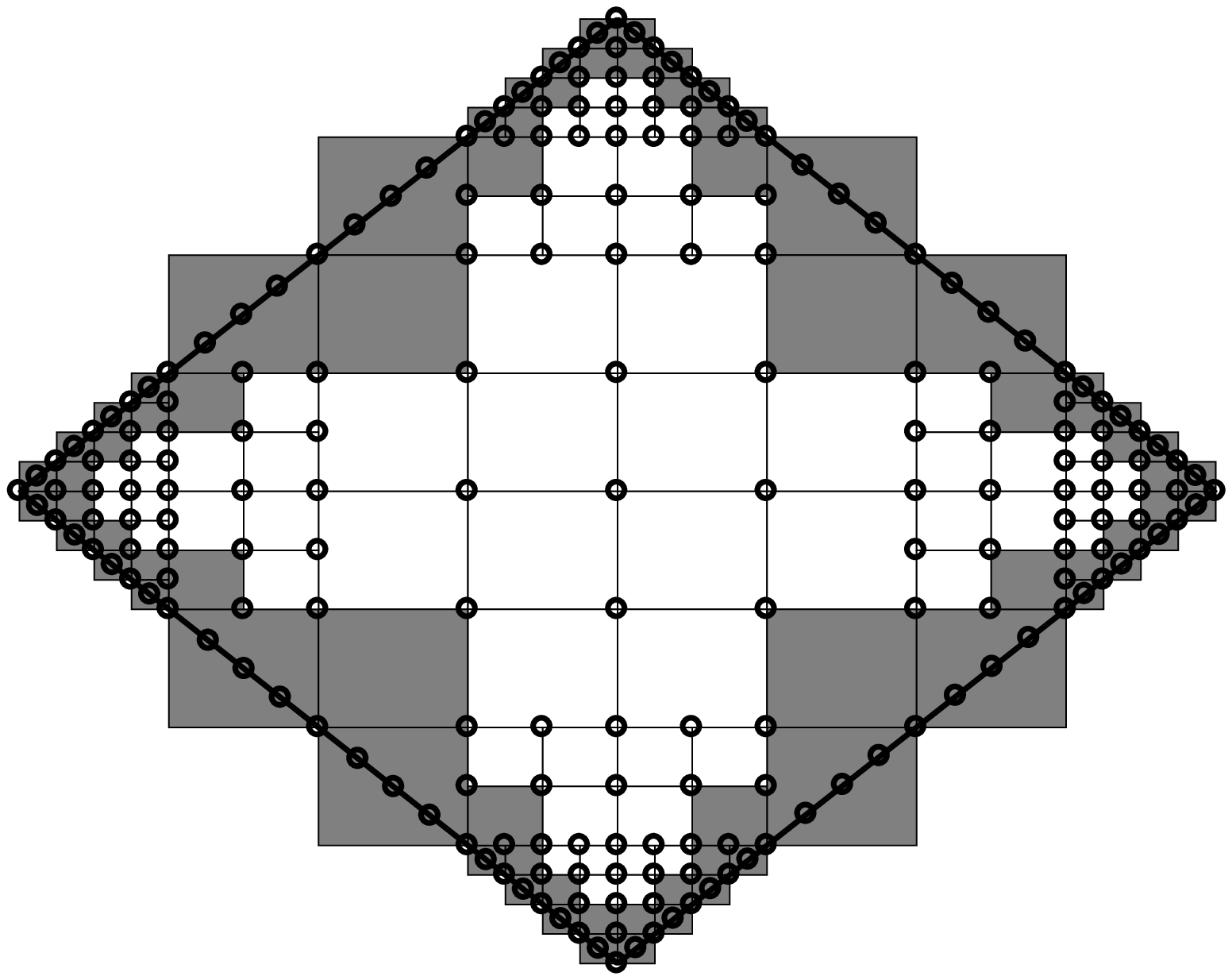}\label{fig:balanced}}
\caption{\emph{A priori} refinement near the corners of the domain: \subref{fig:unbalanced}~unbalanced quadtree and \subref{fig:balanced}~its balanced version.}
\label{fig:RefinementQuadtree}
\end{figure}

\subsection{Generating the stencil}\label{sec:meshes_stencilsearch}

We explain how quadtrees are used to efficiently find the neighbours for each interior mesh point. The main idea is the following: given the quadtree structure we know the relative position of the mesh points and can significantly restrict the number of nodes we examine.

Consider a direction $\nu = (\cos\theta,\sin\theta)$ and the line $x_0 + t\nu$. Without loss of generality, assume the line has positive slope as in Figure~\ref{fig:NeighboursSearch}.  We describe the procedure for finding the required neighbours of $x_0$ lying in the first and fourth quadrants.
\begin{algorithm}[h]
\caption{Finding the neighbours of $x_0 \in \G$ in the first and fourth quadrant.}
\label{alg:stencil_search}
\begin{algorithmic}[1]
\STATE Identify the leaf square that has $x_0$ as its southwest vertex.  This can be done efficiently since, when constructing the quadtree, a record is maintained of the (four) leaf squares that have each interior $x_0$ as a vertex.
\STATE Identify which edge(s) of this square intersect the line $x_0 + t\nu$, selecting the edge that yields the smaller value of $t$, $t_{min}$ (i.e. the first edge to intersect this ray).
\STATE Identify the neighbouring leaf squares that share this edge, selecting the one that intersects the line $x_0 + t\nu$ at $t = t_{min}$.
\STATE Identify the edge of this square that intersects the line $x_0 + t\nu$ at $t = t_{min}$.
\STATE Consider the two endpoints $y_1, y_2$ of this edge as potential neighbours, one lying in the first quadrant and one in the fourth quadrant.
\STATE Repeat steps 2-5, continually adding nodes to the list of potential neighbours, until the ray $x_0 + t\nu$ exits the search region ($t>\epsilon$) or we encounter a boundary leaf square.
\STATE If the procedure terminates at a boundary leaf square, add to the list of potential neighbours all boundary nodes associated with this square.
\end{algorithmic}
\end{algorithm}

From the list of potential neighbours in each quadrant, the precise neighbours used in the stencil are determined via~\eqref{eq:neighbours}. See Figure \ref{fig:NeighboursSearch}, for a close-up of the neighbours search in a uniform (left) and non-uniform (right) mesh.

Referring to Figure~\ref{fig:NeighboursSearch} (a), we provide a rough estimate on the improvement this algorithm yields for a uniform $N\times  N$ grid with grid spacing $h = \bO(1/N)$. Recall that the search region is a disc of radius $\epsilon$. The brute force algorithm used in~\cite{FroeseMeshfree} examines $\bO((\epsilon/h)^2)$ neighbours, while the algorithm proposed above using quadtrees examines only $\bO(\epsilon/h)$ neighbours. Given the typical choice $\epsilon = \bO(\sqrt{h})$, the cost of constructing the stencil at each point is reduced from $\bO(N)$ to $\bO(\sqrt{N})$. A similar speed-up is seen for the piecewise Cartesian meshes produced by the quadtree.


\begin{figure}[h]
\centering
\subfigure[]{\includegraphics[width=\widthtwofigures]{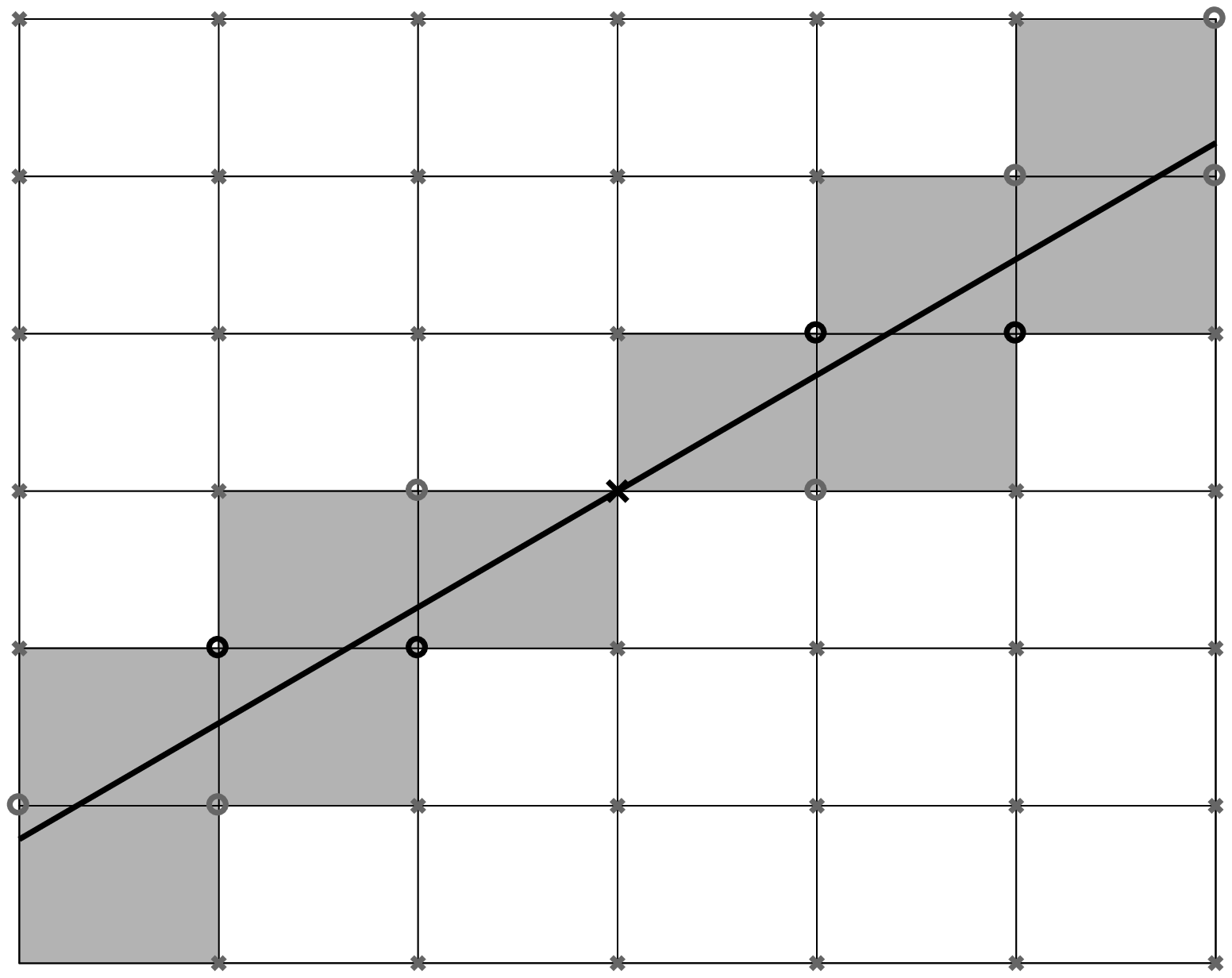}}\label{fig:UniformNeighboursSearch}
\subfigure[]{\includegraphics[width=\widthtwofigures]{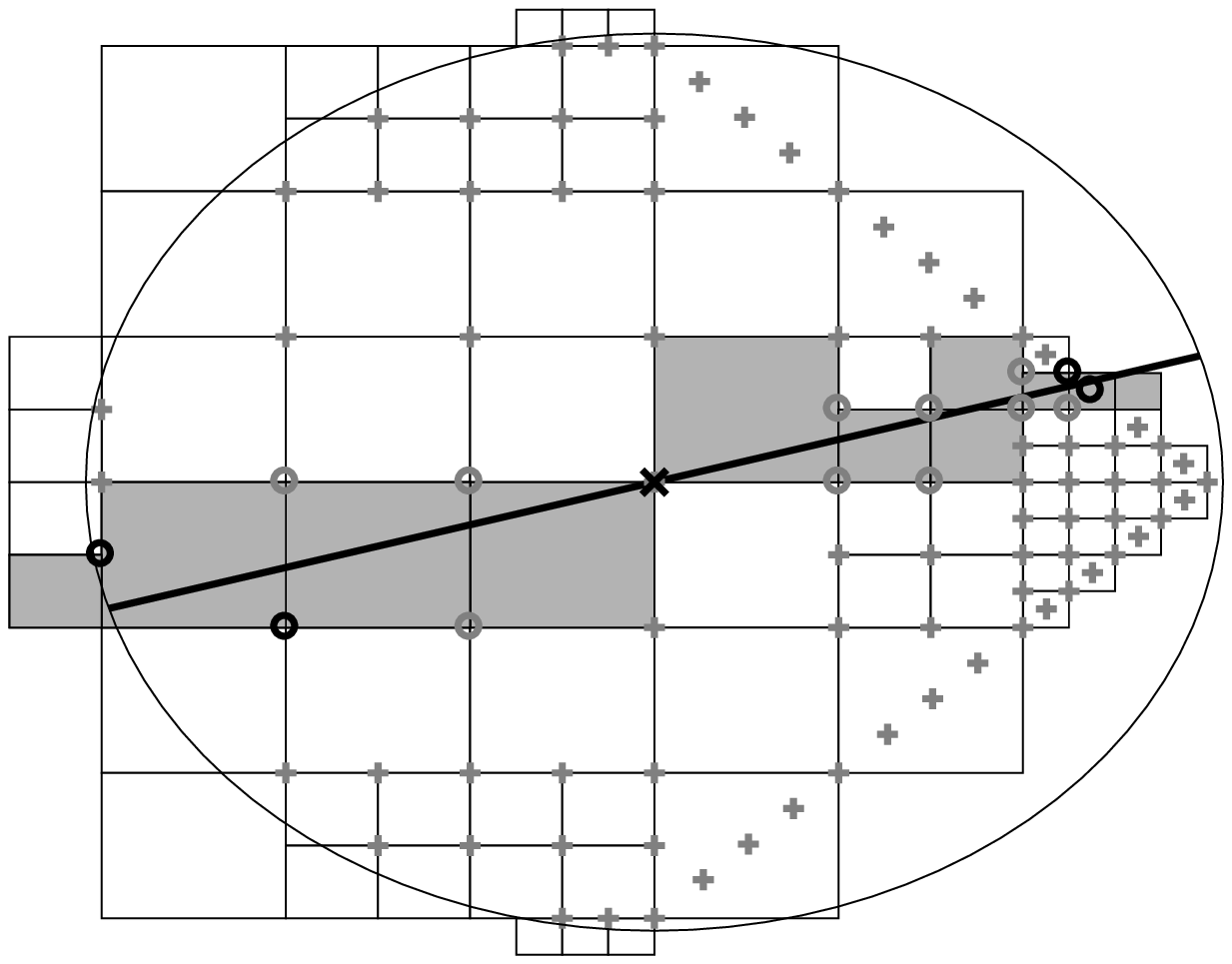}}\label{fig:RefinementNeighboursSearch}
\caption{Potential neighbours of $x_0 \in \G$ (black x-mark) as a result of Algorithm \ref{alg:stencil_search} are marked with a circle, with the selected neighbours in black. All remaining mesh points are marked with an x-mark. The grey squares are the ones considered in Algorithm~\ref{alg:stencil_search}.}
\label{fig:NeighboursSearch}
\end{figure}

%% file: accurate.tex
\section{Higher-Order Methods}\label{sec:accurate}

We also introduce a technique for building formally higher-order approximations on highly non-uniform/unstructured grids.  We focus on second-order schemes, which is typically sufficient for applications, but the same ideas can be easily be extended to higher-order schemes.

\subsection{Filtered schemes}

The meshfree finite difference approximation discussed in \autoref{sec:schemes} is low accuracy; formally it is at best $\bO(\sqrt{h})$. However it can be used as the foundation for higher-order convergent filtered schemes as in~\cite{FOFiltered}. The main idea is to blend a monotone convergent scheme with a non-monotone accurate scheme and retain the advantages of both: stability and convergence of the former, and higher accuracy of the latter. 

To accomplish this, we let $F_A[u]$ be any higher-order scheme and $F_M[u]$ be a monotone approximation scheme, both defined on the same mesh. The filtered scheme is then defined as
\[
F_F[u] = F_M[u] + h^\alpha S\left(\frac{F_A[u]-F_M[u]}{h^\alpha}\right),
\]
where the filter S is given by
\[
S(x) = \begin{cases}
x,		& \abs{x} \leq 1,\\
0,		& \abs{x} \geq 2,\\
-x+2,	& 1 \leq x \leq 2,\\
-x-2,	& -2 \leq x \leq -1.
\end{cases}
\]
As long as $\alpha>0$, this approximation converges to the viscosity solution of the PDE under the same conditions as the monotone scheme converges. The underlying reason is that this scheme is a small perturbation of a monotone scheme and  the proof of the Barles-Souganidis theorem is easily modified to accommodate this. Moreover, if $h^\alpha$ is larger than the discretisation error of the monotone scheme, the formal accuracy of the filtered scheme is the same as the formal accuracy of the non-monotone scheme.

\subsection{Higher-order schemes in interior}

We discuss how to build high-order schemes for the non-uniform meshes proposed in \autoref{sec:meshes}. In this section, we focus on interior mesh points away from the boundary.

Defining higher-order schemes for \eqref{eq:PDE} reduces to defining higher-order approximations to $u_{xx}$, $u_{yy}$ and $u_{xy}$. We will focus on building second order approximations, although the ideas are easily generalised. For a uniform Cartesian grid, such as in Figure \ref{fig:StencilAccurateUniform}, these are widely known and are given by
\begin{align*}
u_{xx} & \approx \frac{u_W+u_E - 2u}{h^2},\\
u_{yy} & \approx \frac{u_N+u_S - 2u}{h^2},\\
u_{xy} & \approx \frac{u_{NE}+u_{SW}-u_{NW}-u_{SE}}{4h^2}.
\end{align*}

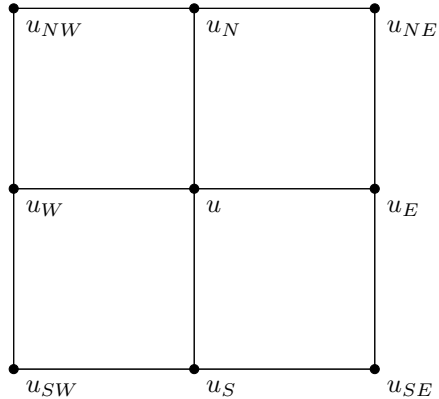
\begin{figure}[h]
\centering
\tikzset{
  latentnode/.style={fill,draw, minimum width=1pt, shape=circle, black}}
\begin{tikzpicture}[scale=0.6]
\draw (0,0) rectangle (8,8); 
\draw (0,4) rectangle (4,8); 
\draw (4,4) rectangle (8,8); 
\draw (0,0) rectangle (4,4); 
\draw (4,0) rectangle (8,4); 
\draw[fill] (0,0) circle (1mm) node[below right=1pt]{$u_{SW}$};
\draw[fill] (4,0) circle (1mm) node[below right=1pt]{$u_{S}$};
\draw[fill] (8,0) circle (1mm) node[below right=1pt]{$u_{SE}$};
\draw[fill] (0,4) circle (1mm) node[below right=1pt]{$u_{W}$};
\draw[fill] (4,4) circle (1mm) node[below right=1pt]{$u$};
\draw[fill] (8,4) circle (1mm) node[below right=1pt]{$u_{E}$};
\draw[fill] (0,8) circle (1mm) node[below right=1pt]{$u_{NW}$};
\draw[fill] (4,8) circle (1mm) node[below right=1pt]{$u_N$};
\draw[fill] (8,8) circle (1mm) node[below right=1pt]{$u_{NE}$};
\end{tikzpicture}
\caption{A regular node and respective stencil for a uniform Cartesian grid.}
\label{fig:StencilAccurateUniform}
\end{figure}

In a uniform cartesian grid, all the nodes are regular nodes; i.e., each node is the vertex of four different squares like the one depicted in Figure \ref{fig:StencilAccurateUniform}. However, the meshes proposed here are non-uniform and in general not all nodes will be regular nodes.  We need also consider dangling nodes, which occur midway along the shared edge of two equally-sized squares, one of which is subdivided. Thus additional work is required to define the higher-order schemes.

As explained in \autoref{sec:meshes}, the meshes are generated using quadtrees that are kept balanced (the lengths of neighbouring squares differ by at most a factor of two). Thus each interior mesh point can be associated to one of five different configurations. These are depicted in Figures \ref{fig:StencilAccurateUniform}, \ref{fig:StencilDanglingNode1}, and \ref{fig:StencilDanglingNode2}. The generic element chosen to represent each configuration is one where each square is a leaf square. In general, one or more of the smaller squares may have children in the quadtree; i.e., they may be subdivided into smaller squares. We consider these to be redundant when constructing the high-order schemes.  Considering all possible different configurations would only increase the complexity of the schemes with no additional benefits as the schemes would remain asymptotically second order; only the asymptotic error constant could be improved. 

For the configurations in Figure \ref{fig:StencilDanglingNode1}, we use the following approximations
\begin{align*}
u_{xx} & \approx \frac{-2u_N-2u_S+u_{NW}+u_{NE}+u_{SE}+u_{SW}}{2h^2},\\
u_{yy} & \approx 4\frac{u_N+u_S - 2u}{h^2},\\
u_{xy} & \approx \frac{u_{NE}+u_{SW}-u_{NW}-u_{SE}}{2h^2}.
\end{align*}
As for the configurations in Figure \ref{fig:StencilDanglingNode2}, we have
\begin{align*}
u_{xx} & \approx 4\frac{u_W+u_E - 2u}{h^2},\\
u_{yy} & \approx \frac{-2u_W-2u_E+u_{NW}+u_{NE}+u_{SE}+u_{SW}}{2h^2},\\
u_{xy} & \approx \frac{u_{NE}+u_{SW}-u_{NW}-u_{SE}}{2h^2}.
\end{align*}
The standard Taylor expansion argument shows that the above expressions are second order accurate.

\twocolumnsfigures{
\begin{tabular}{cc}
\begin{tikzpicture}[scale=0.5]
\draw (0,0) rectangle (4,4);
\draw (4,0) rectangle (6,2);
\draw (4,2) rectangle (6,4);
\draw (6,0) rectangle (8,2);
\draw (6,2) rectangle (8,4);
\draw[fill] (0,0) circle (1mm) node[below right=.1pt]{$u_{SW}$};
\draw[fill] (4,0) circle (1mm) node[below right=.1pt]{$u_{S}$};
\draw[fill] (6,0) circle (1mm);
\draw[fill] (8,0) circle (1mm) node[below right=.1pt]{$u_{SE}$};
\draw[fill] (4,2) circle (1mm) node[below right=.1pt]{$u$};
\draw[fill] (6,2) circle (1mm);
\draw[fill] (8,2) circle (1mm);
\draw[fill] (0,4) circle (1mm) node[below right=.1pt]{$u_{NW}$};
\draw[fill] (4,4) circle (1mm) node[below right=.1pt]{$u_N$};
\draw[fill] (6,4) circle (1mm);
\draw[fill] (8,4) circle (1mm) node[below right=.1pt]{$u_{NE}$};
\end{tikzpicture}\\
\begin{tikzpicture}[scale=0.5]
\draw (4,0) rectangle (8,4);
\draw (0,0) rectangle (2,2);
\draw (0,2) rectangle (2,4);
\draw (2,0) rectangle (4,2);
\draw (2,2) rectangle (4,4);
\draw[fill] (0,0) circle (1mm) node[below right=.1pt]{$u_{SW}$};
\draw[fill] (4,0) circle (1mm) node[below right=.1pt]{$u_{S}$};
\draw[fill] (2,0) circle (1mm);
\draw[fill] (8,0) circle (1mm) node[below right=.1pt]{$u_{SE}$};
\draw[fill] (0,2) circle (1mm);
\draw[fill] (2,2) circle (1mm);
\draw[fill] (4,2) circle (1mm) node[below right=.1pt]{$u$};
\draw[fill] (0,4) circle (1mm) node[below right=.1pt]{$u_{NW}$};
\draw[fill] (4,4) circle (1mm) node[below right=.1pt]{$u_N$};
\draw[fill] (2,4) circle (1mm);
\draw[fill] (8,4) circle (1mm) node[below right=.1pt]{$u_{NE}$};
\end{tikzpicture}
\end{tabular}
\captionsetup{width=\textwidth}
\caption{A dangling node in the $x$ variable and respective stencil for the higher-order scheme.}
\label{fig:StencilDanglingNode1}
}{
\begin{tabular}{cc}
\begin{tikzpicture}[scale=0.5]
\draw (0,0) rectangle (4,4);
\draw (0,6) rectangle (2,8);
\draw (2,6) rectangle (4,8);
\draw (0,4) rectangle (2,6);
\draw (2,4) rectangle (4,6);
\draw[fill] (0,0) circle (1mm) node[below right=.1pt]{$u_{SW}$};
\draw[fill] (4,0) circle (1mm) node[below right=.1pt]{$u_{SE}$};
\draw[fill] (0,4) circle (1mm) node[below right=.1pt]{$u_{W}$};
\draw[fill] (2,4) circle (1mm) node[below right=.1pt]{$u$};
\draw[fill] (4,4) circle (1mm) node[below right=.1pt]{$u_{E}$};
\draw[fill] (0,6) circle (1mm);
\draw[fill] (2,6) circle (1mm);
\draw[fill] (4,6) circle (1mm);
\draw[fill] (0,8) circle (1mm) node[below right=.1pt]{$u_{NW}$};
\draw[fill] (2,8) circle (1mm);
\draw[fill] (4,8) circle (1mm) node[below right=.1pt]{$u_{NE}$};
\end{tikzpicture} & 
\begin{tikzpicture}[scale=0.5]
\draw (0,2) rectangle (2,4);
\draw (2,2) rectangle (4,4);
\draw (0,0) rectangle (2,2);
\draw (2,0) rectangle (4,2);
\draw (0,4) rectangle (4,8);
\draw[fill] (0,0) circle (1mm) node[below right=.1pt]{$u_{SW}$};
\draw[fill] (2,0) circle (1mm);
\draw[fill] (4,0) circle (1mm) node[below right=.1pt]{$u_{SE}$};
\draw[fill] (0,2) circle (1mm);
\draw[fill] (2,2) circle (1mm);
\draw[fill] (4,2) circle (1mm);
\draw[fill] (0,4) circle (1mm) node[below right=.1pt]{$u_{W}$};
\draw[fill] (2,4) circle (1mm) node[below right=.1pt]{$u$};
\draw[fill] (4,4) circle (1mm) node[below right=.1pt]{$u_{E}$};
\draw[fill] (0,8) circle (1mm) node[below right=.1pt]{$u_{NW}$};
\draw[fill] (4,8) circle (1mm) node[below right=.1pt]{$u_{NE}$};
\end{tikzpicture}
\end{tabular}
\captionsetup{width=\textwidth}
\caption{A dangling node in the $y$ variable and respective stencil for the higher-order scheme.}
\label{fig:StencilDanglingNode2}
}

Finally, we explain how one can efficiently determine the configuration of each interior mesh point. As mentioned in \autoref{sec:meshes_stencilsearch}, for each interior mesh point a record is kept of the four leaf squares that the interior mesh point as a vertex. Thus the configuration is easily determined by determining the depth of the neighbouring squares and respective parent squares in the quadtree.

\begin{figure}
	
\end{figure}

\subsection{Least-squares constructions near boundary}

In this section, we discuss how to construct second order schemes at interior points near the boundary. When near the boundary, the construction of the schemes cannot reduce to the cases considered in the previous section: in general, not all the neighbouring mesh points will be the vertices of squares and some will lie on the boundary, which we allow to have a complicated geometry. Thus additional care is needed.  Here we describe a general strategy for building high-order schemes.

Let $\{x_i\}_{i=1}^N$ denote neighbouring mesh points to the interior mesh point $x_0$ with $\Norm{x_i-x_0}_\infty = \bO(h)$. Using Taylor expansion we obtain, for each $i=1,\ldots,N$,
\[
u(x_i) - u(x_0) = \sum_{0 < \abs{\alpha}\leq 3} \frac{(x_i-x_0)^\alpha}{\alpha!} (\partial^\alpha u)(x_0) + \bO(h^4),
\]
where we are using the multi-index notation. Hence
\[
\sum_{i=1}^N a_i(u(x_i)-u(x_0)) = \partial^\beta u(x_0) + \bO(h^2)
\]
if the $\{a_i\}_{i=1}^N$ solve the linear system
\[
\sum_{i=1}^N \frac{(x_i-x_0)^\alpha}{\alpha!}a_i = \mathbb{1}_{\{\alpha = \beta\}}
\]
for $0 < \abs{\alpha}\leq 3$.

To approximate second derivatives to second order, we expect to require $N=9$ neighbours.

Designing second order schemes is now reduced to determining the neighbouring mesh points and solving the respective linear system. However, since we are particularly interested in the case where some of the neighbouring mesh points lie on the boundary, which may have a complicated geometry, it is hard to make any \emph{a priori} claim regarding the invertability and conditioning of the linear system. It is important to point out that we are interested in obtaining any particular solution. As we saw in the previous section, depending on the derivative being approximated and the location of the neighbouring mesh points, the number of neighbouring mesh points required changes.

We are now ready to describe the strategy implemented to construct the higher-order schemes. First, we determine which configuration we are in. If all the vertices of squares neighbouring $x_0$ are mesh points, we use the approximations described in the previous section. Otherwise, we build the linear system above using all mesh points $x_i$ that lie within squares adjoining $x_0$ (some will lie on the boundary of the domain and will not be vertices of the squares). In general, we will have $N \geq 9$, but the linear system may still have no solution or be ill-conditioned. If that is the case, we consider additional neighbouring mesh points by adding the mesh points that belong to neighbouring squares. In general, we end up with an under-determined system and we select the least squares solution. In practice, adding additional neighbouring mesh points was not always required, and it was never required more than once. Thus the high-order scheme has stencil width $\bO(h)$ and preserves the formal discretisation error of $\bO(h^2)$.

%% file: Examples.tex
\section{Computational Examples}\label{sec:examples}

\subsection{Monge-Amp\`ere equation}

We consider the Monge-Amp\`ere equation
\[
\begin{cases}
-\det(D^2u)+f = 0,	& x \in \Omega\\
u = g,				& x \in \partial \Omega\\
\text{$u$ is convex.}
\end{cases}
\]
The PDE is only elliptic in the space of convex functions. However, as in \cite{FroeseTransport}, we can use the globally elliptic extension
\[\hspace{-5ex}
-\min_{\theta \in [0,\pi/2)} \left\{\max\left\{\frac{\partial^2 u}{\partial e_\theta^2},0\right\}\max\left\{\frac{\partial^2 u}{\partial e_{\theta+\pi/2}^2},0\right\} + \min\left\{\frac{\partial^2 u}{\partial e_\theta^2},0\right\} + \max\left\{\frac{\partial^2 u}{\partial e_{\theta+\pi/2}^2},0\right\} \right\} + f = 0.
\]

We will consider four different domains given by $\Omega = \{(x,y)\in\R^2\mid \phi(x,y) < 0\}$ where $\phi$ is given by
\begin{enumerate}[(a)]\setlength\itemsep{0.8em}
	\item (circle) $\phi(x,y) = x^2+y^2-1$,
	\item (ellipse) $\phi(x,y) = x^2+2y^2-1$,
	\item (diamond) $\phi(x,y) = |x|+|y|-1$,
	\item (diamond stretched) $\phi(x,y) = |x|+|2y|-1$.
\end{enumerate}

\begin{example}\label{ex:MA_C2}
We consider first the following $C^2$ solution of the Monge-Amp\`ere equation
\[
u(x,y) = e^{\frac{x^2+y^2}{2}}, \quad f(x,y) = (1+x^2+y^2)e^{x^2+y^2}.
\]
\end{example}

Results are displayed in Table~\ref{table:ExampleMA_C2}.  On each domain, the filtered implementation recovers the desired second-order accuracy even though the boundary nodes do not belong to the structured piecewise Cartesian mesh.

\begin{table}[htp]
\centering
\footnotesize
\begin{tabular}{cccccc}
\hline
	& \multicolumn{5}{l}{Errors and order, Example \ref{ex:MA_C2} (circle)} \\ \cline{2-6}
$N$	& $h$		& \multicolumn{2}{c}{Monotone}	& \multicolumn{2}{c}{Filtered}\\
\hline
101 & \num{2.500e-01} & \num{1.739e-02} & - & \num{7.008e-03} & - \\
349 & \num{1.250e-01} & \num{5.690e-03} & 1.61 & \num{2.219e-03} & 1.66 \\
1137 & \num{6.250e-02} & \num{2.780e-03} & 1.03 & \num{6.184e-04} & 1.84 \\
4289 & \num{3.125e-02} & \num{1.876e-03} & 0.57 & \num{1.708e-04} & 1.86 \\
15685 & \num{1.562e-02} & \num{1.630e-03} & 0.20 & \num{4.320e-05} & 1.98 \\
58449 & \num{7.812e-03} & \num{1.567e-03} & 0.06 & \num{1.082e-05} & 2.00 \\
	& \multicolumn{5}{l}{Errors and order, Example \ref{ex:MA_C2} (ellipse)} \\ \cline{2-6}
$N$	& $h$		& \multicolumn{2}{c}{Monotone}	& \multicolumn{2}{c}{Filtered}\\
\hline
79 & \num{2.500e-01} & \num{9.467e-03} & - & \num{8.884e-03} & - \\
257 & \num{1.250e-01} & \num{3.026e-03} & 1.65 & \num{1.233e-03} & 2.85 \\
893 & \num{6.250e-02} & \num{1.262e-03} & 1.26 & \num{3.628e-04} & 1.76 \\
3143 & \num{3.125e-02} & \num{8.277e-04} & 0.61 & \num{1.012e-04} & 1.84 \\
11257 & \num{1.562e-02} & \num{7.443e-04} & 0.15 & \num{2.602e-05} & 1.96 \\
42863 & \num{7.812e-03} & \num{7.238e-04} & 0.04 & \num{6.512e-06} & 2.00 \\
	& \multicolumn{5}{l}{Errors and order, Example \ref{ex:MA_C2} (diamond)} \\ \cline{2-6}
$N$	& $h$		& \multicolumn{2}{c}{Monotone}	& \multicolumn{2}{c}{Filtered}\\
\hline
57 & \num{2.500e-01} & \num{6.162e-03} & - & \num{7.244e-03} & -\\
177 & \num{1.250e-01} & \num{2.332e-03} & 1.40 & \num{1.478e-03} & 2.29 \\
673 & \num{6.250e-02} & \num{9.716e-04} & 1.26 & \num{3.561e-04} & 2.05 \\
2369 & \num{3.125e-02} & \num{3.413e-04} & 1.51 & \num{8.881e-05} & 2.00 \\
9089 & \num{1.562e-02} & \num{2.771e-04} & 0.30 & \num{2.218e-05} & 2.00 \\
35073 & \num{7.812e-03} & \num{1.878e-04} & 0.56 & \num{5.544e-06} & 2.00 \\
	& \multicolumn{5}{l}{Errors and order, Example \ref{ex:MA_C2} (diamond stretched)} \\ \cline{2-6}
$N$	& $h$		& \multicolumn{2}{c}{Monotone}	& \multicolumn{2}{c}{Filtered}\\
\hline
61 & \num{2.500e-01} & \num{2.535e-03} & - & \num{1.888e-03} & - \\
153 & \num{1.250e-01} & \num{7.891e-04} & 1.68 & \num{5.212e-04} & 1.86 \\
497 & \num{6.250e-02} & \num{2.802e-04} & 1.49 & \num{1.356e-04} & 1.94 \\
1633 & \num{3.125e-02} & \num{1.516e-04} & 0.89 & \num{3.428e-05} & 1.98 \\
5569 & \num{1.562e-02} & \num{7.684e-05} & 0.98 & \num{8.597e-06} & 2.00 \\
20353 & \num{7.812e-03} & \num{2.357e-05} & 1.70 & \num{2.151e-06} & 2.00 \\
\end{tabular}
\caption{Convergence results for the $C^2$ solution of the Monge-Amp\`ere
equation.}
\label{table:ExampleMA_C2}
\end{table}

\begin{example}\label{ex:MA_C1}
We consider also a $C^1$ solution of the Monge-Amp\`ere equation, for which the ellipticity is degenerate in an open set.
\[
u(x,y) = \frac{1}{2}\max\{\sqrt{x.^2+y.^2}-0.2,0\}^2, \quad f(x,y) = \max\{1-\frac{0.2}{\sqrt{x^2+y^2)}},0\}.
\]
\end{example}

This solution is singular, and there is thus no realistic hope of attaining the formal second-order discretisation error obtained from Taylor's Theorem.  Nevertheless, the method converges and we observe roughly first-order accuracy (Table~\ref{table:ExampleMA_C1}).

\begin{table}[htp]
\centering
\footnotesize
\begin{tabular}{cccccc}
\hline
	& \multicolumn{5}{l}{Errors and order, Example \ref{ex:MA_C1} (circle)} \\ \cline{2-6}
$N$	& $h$		& \multicolumn{2}{c}{Monotone}	& \multicolumn{2}{c}{Filtered}\\
\hline
101 & \num{2.500e-01} & \num{1.606e-02} & - & \num{4.142e-03} & - \\
349 & \num{1.250e-01} & \num{9.106e-03} & 0.82 & \num{2.384e-03} & 0.80 \\
1137 & \num{6.250e-02} & \num{5.706e-03} & 0.67 & \num{1.763e-03} & 0.44 \\
4289 & \num{3.125e-02} & \num{4.827e-03} & 0.24 & \num{7.871e-04} & 1.16 \\
15685 & \num{1.562e-02} & \num{4.330e-03} & 0.16 & \num{3.674e-04} & 1.10 \\
58449 & \num{7.812e-03} & \num{4.300e-03} & 0.01 & \num{1.839e-04} & 1.00 \\
	& \multicolumn{5}{l}{Errors and order, Example \ref{ex:MA_C1} (ellipse)} \\ \cline{2-6}
$N$	& $h$		& \multicolumn{2}{c}{Monotone}	& \multicolumn{2}{c}{Filtered}\\
\hline
79 & \num{2.500e-01} & \num{1.193e-02} & - & \num{4.276e-03} & - \\
257 & \num{1.250e-01} & \num{7.126e-03} & 0.74 & \num{3.201e-03} & 0.42 \\
893 & \num{6.250e-02} & \num{4.569e-03} & 0.64 & \num{1.579e-03} & 1.02 \\
3143 & \num{3.125e-02} & \num{3.924e-03} & 0.22 & \num{8.118e-04} & 0.96 \\
11257 & \num{1.562e-02} & \num{3.519e-03} & 0.16 & \num{3.764e-04} & 1.11 \\
42863 & \num{7.812e-03} & \num{3.497e-03} & 0.01 & \num{2.224e-04} & 0.76 \\
	& \multicolumn{5}{l}{Errors and order, Example \ref{ex:MA_C1} (diamond)} \\ \cline{2-6}
$N$	& $h$		& \multicolumn{2}{c}{Monotone}	& \multicolumn{2}{c}{Filtered}\\
\hline
57 & \num{2.500e-01} & \num{8.503e-03} & - & \num{6.576e-03} & - \\
177 & \num{1.250e-01} & \num{6.670e-03} & 0.35 & \num{2.155e-03} & 1.61 \\
673 & \num{6.250e-02} & \num{4.265e-03} & 0.65 & \num{1.555e-03} & 0.47 \\
2369 & \num{3.125e-02} & \num{3.442e-03} & 0.31 & \num{7.661e-04} & 1.02 \\
9089 & \num{1.562e-02} & \num{2.303e-03} & 0.58 & \num{3.597e-04} & 1.09 \\
35073 & \num{7.812e-03} & \num{1.109e-03} & 1.05 & \num{2.142e-04} & 0.75 \\
	& \multicolumn{5}{l}{Errors and order, Example \ref{ex:MA_C1} (diamond stretched)} \\ \cline{2-6}
$N$	& $h$		& \multicolumn{2}{c}{Monotone}	& \multicolumn{2}{c}{Filtered}\\
\hline
61 & \num{2.500e-01} & \num{2.653e-03} & - & \num{4.060e-03} & -\\
153 & \num{1.250e-01} & \num{3.780e-03} & -0.51 & \num{3.026e-03} & 0.42 \\
497 & \num{6.250e-02} & \num{2.426e-03} & 0.64 & \num{1.213e-03} & 1.32 \\
1633 & \num{3.125e-02} & \num{1.820e-03} & 0.41 & \num{1.228e-03} & -0.02 \\
5569 & \num{1.562e-02} & \num{1.203e-03} & 0.60 & \num{3.726e-04} & 1.72 \\
20353 & \num{7.812e-03} & \num{6.089e-04} & 0.98 & \num{2.015e-04} & 0.89 \\
\end{tabular}
\caption{Convergence results for the $C^1$ solution of the Monge-Amp\`ere
equation.}
\label{table:ExampleMA_C1}
\end{table}

\subsection{Computation time}

One of the reasons to use quadtrees to build non-uniform meshes was to efficiently find the neighbours for the finite difference schemes. Here we compare the implementation with quadtrees, discussed in depth in \autoref{sec:meshes}, with the simple but inefficient brute force approach of~\cite{FroeseMeshfree}.  Reduced CPU time is the ultimate goal, but not necessarily a robust measure of efficiency as it depends on both hardware and software implementations. Here both implementations were vectorised whenever possible in order to optimise the code for MATLAB. Although it is still possible to further optimise the code, the CPU time should  be a fair indication of the improvement gained by using quadtrees.

In Figure~\ref{fig:TimeComparison}, we compare the number of mesh points versus the CPU time required to generate the mesh and find the stencil. We let the domain $\Omega$ be the ellipse given by $\Omega = \{(x,y)\in\R^2\mid x^2+2y^2 < 1\}$. The results indicate that our new approach is optimal, with the computation time required to construct the stencils being roughly proportional to the number of mesh points.  This represents an improvement of roughly one order over the original brute force approach. In practice, we can generate a mesh and find the stencil based on a uniform $256 \times 256$ grid in roughly $50$ seconds with the use of quadtrees, instead of over $6$ minutes with the previous brute force approach.

\begin{figure}[h]
\centering
\begin{tabular}{cc}
\includegraphics[width=\widthtwofigures]{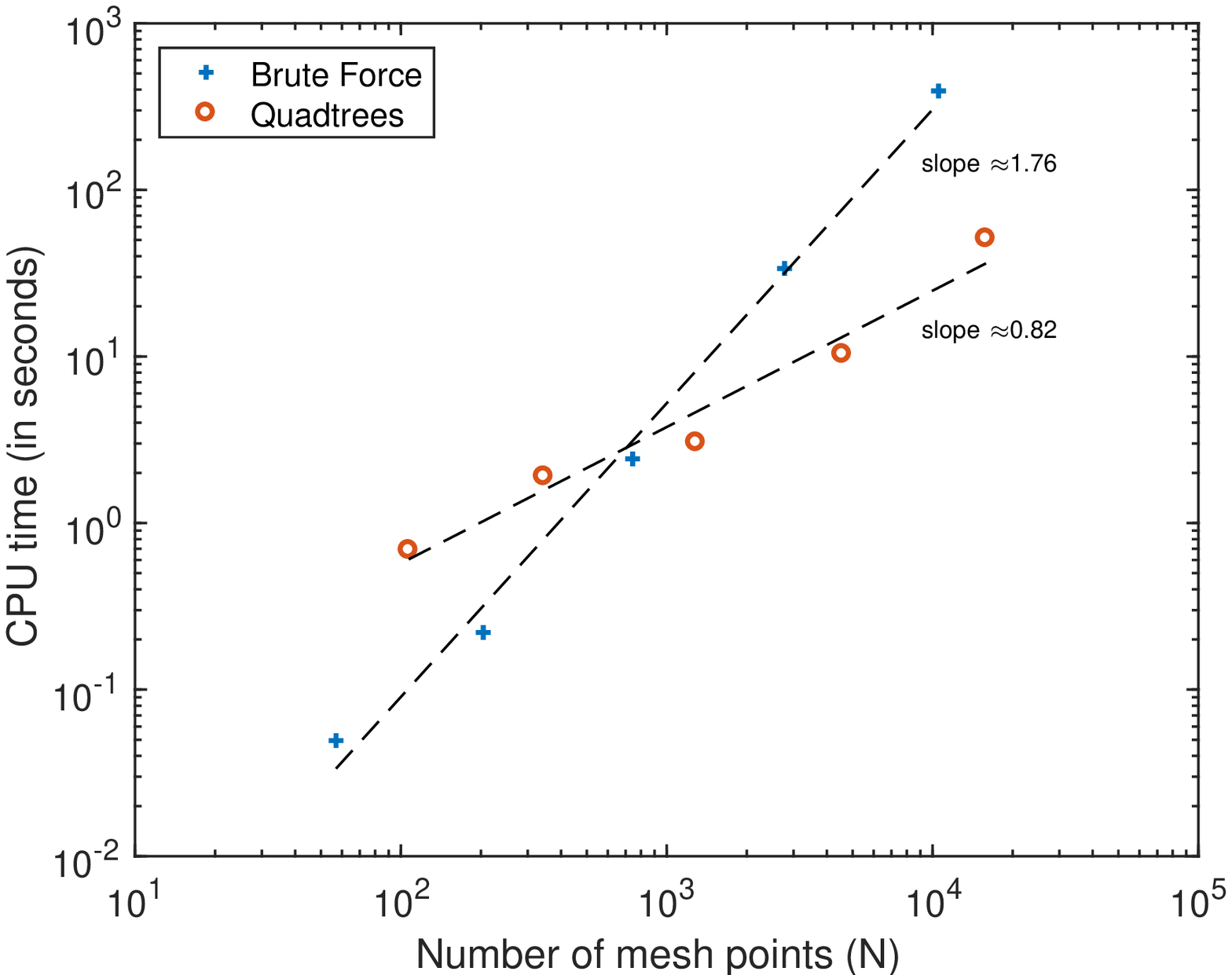}
\end{tabular}
\caption{Number of mesh points vs CPU time in seconds to generate a mesh and find the respective stencil for an ellipsoidal domain.}
\label{fig:TimeComparison}
\end{figure}

\subsection{Adaptivity}

In our next example, we demonstrate the improvements possible with adaptivity using our generalised finite difference approximations.

\begin{example}\label{ex:convenv_adaptivity}
We consider the fully nonlinear convex envelope equation
\[
\begin{cases}
\max\left\{-\lambda_-(D^2u),u-g\right\} = 0,		& x \in \Omega\\
u = 0.5,										& x \in \partial \Omega,
\end{cases}
\]
where
\[
\lambda_-(D^2u) = \min_{\theta \in [0,2\pi]} \frac{\partial^2 u}{\partial e_\theta^2}.
\]

The equation is posed on an ellipse with semi-major axis equal to one and semi-minor axis equal to one-half, which is rotated through an angle of $\phi = \pi/6$. The obstacle $g$ consists of two cones,
\begin{align*}
g_1(x,y) = (x\cos\phi+y \sin\phi+0.5)^2 +(-x\sin\phi+y\cos\phi)^2\\
g_2(x,y) = (x\cos\phi+y \sin\phi-0.5)^2 +(-x\sin\phi+y\cos\phi)^2\\
g(x,y) = \min\left\{g_1(x,y), g_2(x,y), 0.5\right\}
\end{align*}
and the exact solution is
\[
u(x,y) = \begin{cases}
\min\left\{g_1(x,y),g_2(x,y)\right\},	& \abs{x\cos\phi+y\sin\phi} \geq 0.5\\
\abs{-x\sin\phi+y\cos\phi},				& \abs{x\cos\phi+y\sin\phi} < 0.5.
\end{cases}
\]
We note that this solution is only Lipschitz continuous, and the equation must be understood in a weak sense.
\end{example}

We defined the following refinement strategy. Given a solution a solution $u^h$,
\begin{enumerate}
	\item Compute numerically $\Norm{D^2u} = \sqrt{u_{xx}^2+2u_{xy}^2+u_{yy}^2}$.
	\item Refine twice every square with a vertex such that $h \Norm{D^2u} > 0.5$.
\end{enumerate}
This refinement strategy is inspired by the hybrid scheme proposed in \cite{ObermanCENumerics}.

A convergence plot is displayed in Figure~\ref{fig:ExampleAdaptivity}.  We observe an improvement in accuracy from roughly $\bO(N^{-0.38})$ to $\bO(N^{-0.64})$.  Moreover, the use of adaptivity leads to a clear qualitative improvement in the computed solutions.  This is evident in Figure~\ref{fig:ExampleAdaptivitySolutions}, which shows that the solution obtained with the uniform mesh is visibly non-convex along the singularity, which does not align with the grid.  By resolving this singularity, the adaptive method  produces a solution that has a dramatically better quality.

\begin{figure}[h]
\centering
\begin{tabular}{cc}
\includegraphics[width=\widthtwofigures]{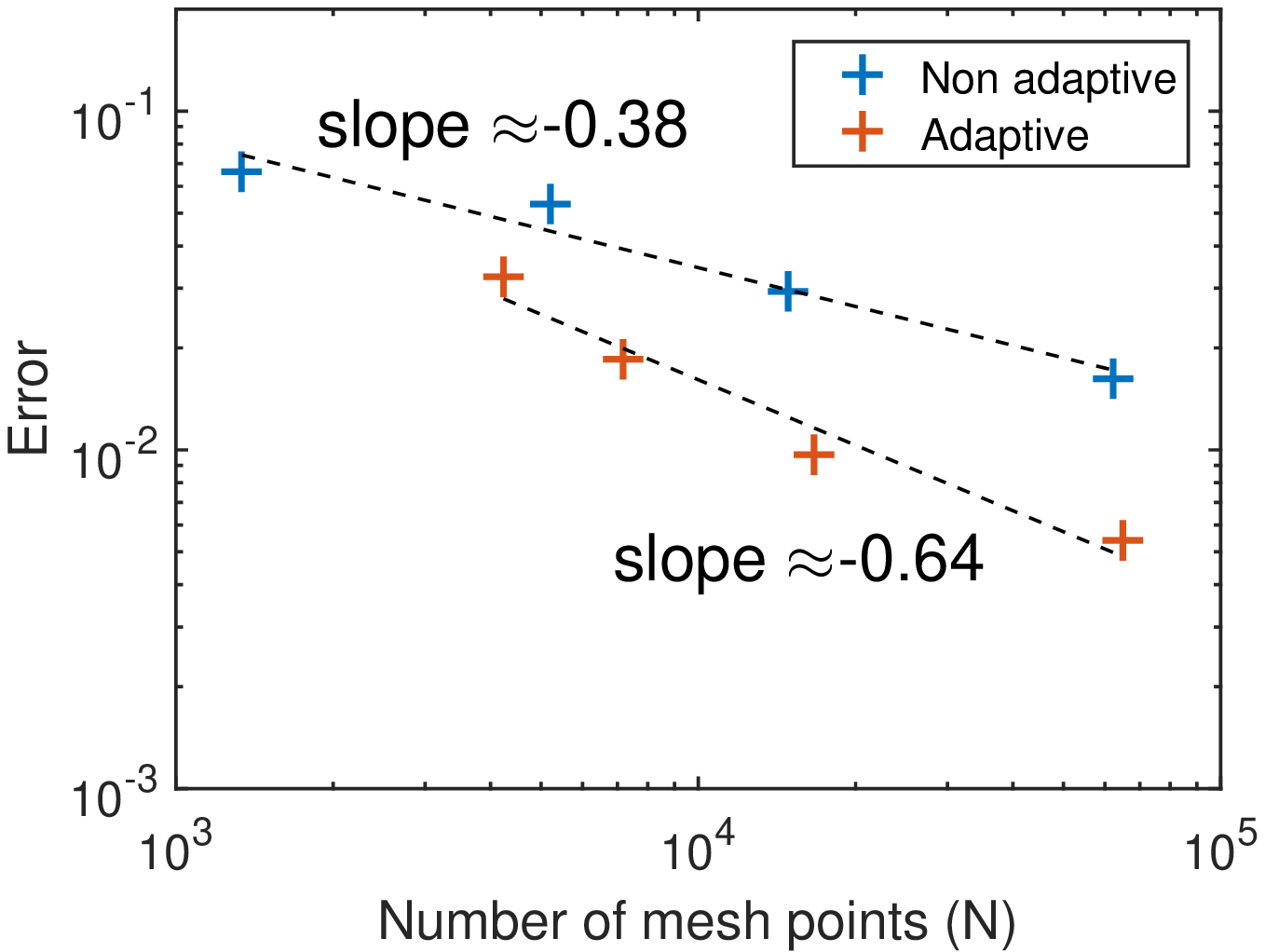}
\end{tabular}
\caption{Number of mesh points vs error for Example \ref{ex:convenv_adaptivity}.}
\label{fig:ExampleAdaptivity}
\end{figure}

\begin{figure}[h]
\centering
\begin{tabular}{cc}
\subfigure[]{\includegraphics[width=\widthtwofigures]{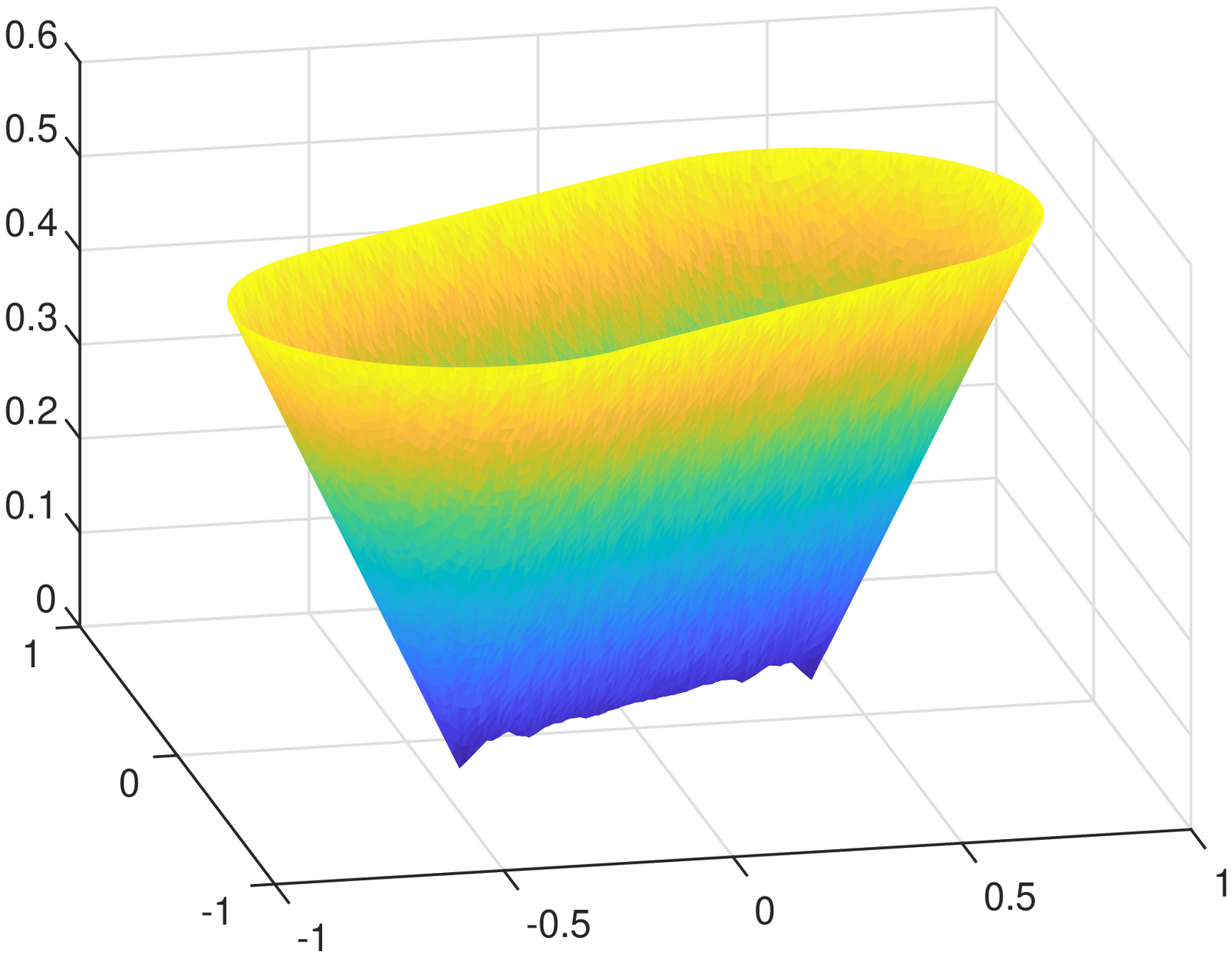}\label{fig:nonadaptive}}
\subfigure[]{\includegraphics[width=\widthtwofigures]{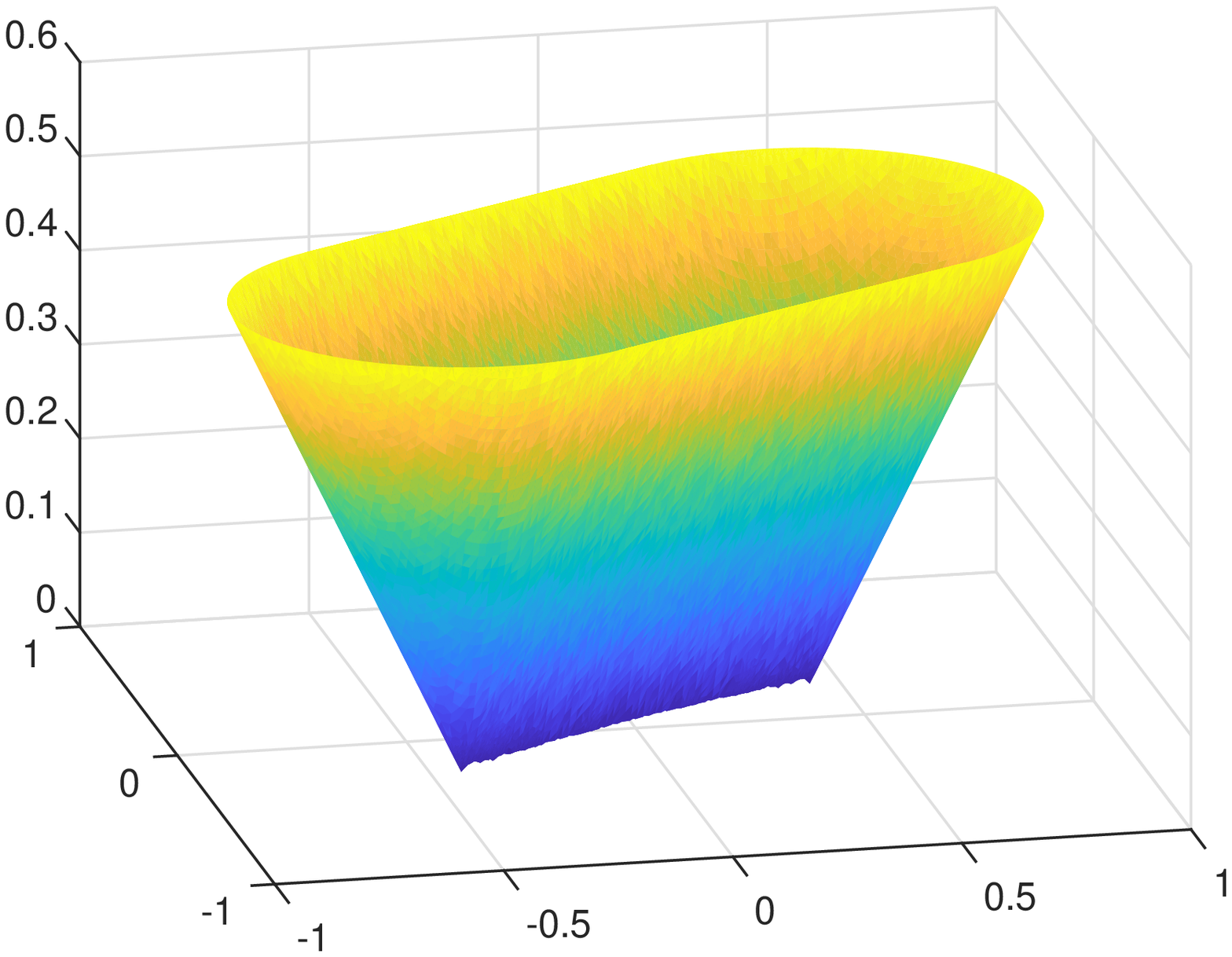}\label{fig:adaptive}}
\end{tabular}
\caption{Solutions of the convex envelope equation computed with the non-adaptive (left) and adaptive approaches (right).}
\label{fig:ExampleAdaptivitySolutions}
\end{figure}

%% file: conclusions.tex
\section{Conclusions}\label{sec:conclusions}
In this article, we described generalised finite difference methods for solving a large class of fully nonlinear elliptic partial differential equations.  These methods were inspired by the meshfree methods described in~\cite{FroeseMeshfree}, which are flexible and convergent, but very expensive to implement.

Our meshes used a modified quadtree structure that used piecewise Cartesian grids in the interior of the domain, augmented by a set of points lying exactly on the boundary.  The inclusion of the additional boundary points was needed to ensure convergence of the numerical methods.  This type of mesh also allows us to deal easily with complicated boundary geometries.

By relying on the underlying quadtree structure, we developed an algorithm for efficiently constructing the mesh and finite difference stencils.  This led to a dramatic improvement in computational efficiency as compared to the original brute force approach.

We also described a strategy for constructing higher-order approximations, which still fit within the convergence framework.  In the interior of the domain, the quadtree structure allows us to explicitly write out the higher-order finite difference schemes.  This strategy fails near the boundary, which can have a very complicated geometry.  At these points, we employed a least-squares approach to construct higher-order schemes.

We also used these methods to perform automatic mesh adaptation, which refines the mesh near singularities in the computed solutions.  This led to a dramatic qualitative and quantitative improvement in computed solutions.

\section*{Acknowledgments}
We thank Adam Oberman for helpful discussions and support of this project.